\documentclass[12pt, a4paper, parskip=half, abstracton, bibliography=totoc]{scrartcl}
\pdfoutput=1

\usepackage{array}
\usepackage{marginnote}
\usepackage{xcolor}
\usepackage{amscd,amssymb,amsfonts,amsmath,latexsym,amsthm}
\usepackage{hyperref}
\usepackage[all,cmtip]{xy}
\usepackage{mathrsfs}
\usepackage{graphicx}
\usepackage{bm} 
\usepackage{etoolbox}
\def\hB{\hspace*{\fill}$\qed$}
\usepackage{mathtools} 

\usepackage{caption}

\let\counterwithout\relax

\usepackage{chngcntr} 

\usepackage{defs_pp1}
\usepackage{slashed}
\usepackage[utf8]{inputenc}
\usepackage{microtype}
\usepackage[english]{babel}

\title{Coarse assembly maps}

\date{\today}

\author{
Ulrich Bunke\thanks{Fakult{\"a}t f{\"u}r Mathematik,
Universit{\"a}t Regensburg,
93040 Regensburg,
GERMANY\newline
ulrich.bunke@mathematik.uni-regensburg.de} 
\and
Alexander Engel\thanks{Fakult{\"a}t f{\"u}r Mathematik,
Universit{\"a}t Regensburg,
93040 Regensburg,
GERMANY\newline
alexander.engel@mathematik.uni-regensburg.de}
}
\numberwithin{equation}{section}
\counterwithout{footnote}{section}

\newtheorem{theorem}{Theorem}[section] 
\newtheorem{prop}[theorem]{Proposition}
\newtheorem{lem}[theorem]{Lemma}
\newtheorem{ddd}[theorem]{Definition}
\newtheorem{kor}[theorem]{Corollary}

\theoremstyle{remark}

\theoremstyle{definition}

\newtheorem{ex}[theorem]{Example}
\newtheorem{rem}[theorem]{Remark}

\newcommand{\BT}{\mathbf{\TopBorn}}
\newcommand{\ubs}{uniform bornological coarse space }

\newcommand{\bP}{\mathbf{P}}

\newcommand{\UBC}{\mathbf{UBC}}

\newcommand{\Yo}{\mathrm{Yo}}

\newcommand{\BC}{\mathbf{BornCoarse}}
\newcommand{\TopBorn}{\mathbf{TopBorn}}

\newcommand{\Cofib}{\mathrm{Cofib}}

\newcommand{\cP}{\mathcal{P}}

\newcommand{\bF}{{\mathbf{F}}}

\newcommand{\PSh}{{\mathbf{PSh}}}

\newcommand{\const}{{\mathtt{const}}}

\newcommand{\cO}{{\mathcal{O}}}
\newcommand{\cU}{{\mathcal{U}}}
\newcommand{\cY}{{\mathcal{Y}}}

\renewcommand{\colim}{\mathrm{colim}}
\newcommand{\Fl}{{F\hspace{-1pt}l}}

\newcommand{\KX}{K\!\mathcal{X}}

\newcommand{\lf}{{l\! f}}
\newcommand{\an}{\mathit{an}}

\newcommand{\free}{\mathit{free}}
\newcommand{\geom}{\mathit{geom}}
\newcommand{\sm}{\mathit{sm}}

\newcommand{\Born}{\mathbf{Born}}

\newcommand{\Spc}{\mathbf{Spc}}

\newcommand{\IN}{\mathbb{N}}

\newcommand{\IR}{\mathbb{R}}

\begin{document}

\maketitle

\begin{abstract}
For every strong coarse homology theory we construct a coarse assembly map as a natural transformation between coarse homology theories. We provide various conditions implying that this assembly map is an equivalence. These results 
generalize known results for the analytic coarse assembly map for $K$-homology  to general coarse homology theories. Furthermore, we calculate the
domain of the coarse assembly map explicitly in terms of locally finite homology theory.
\end{abstract}

\tableofcontents

\section{Introduction}

In this paper we propose a construction of a coarse assembly map as a natural transformation between coarse homology theories and study conditions  under which it is an equivalence.

The classical instance of the coarse assembly map is the coarse analytic assembly map featuring in the coarse Baum--Connes conjecture (Higson--Roe \cite{hr}). For a proper metric space it is constructed as a homomorphisms of $\Z$-graded groups  from the coarsification of the locally
finite $K$-homology groups of this space to the  coarse $K$-homology groups of the same space.  

In \cite{buen} we refined the domain and the target of the coarse analytic assembly map to spectrum-valued  coarse homology theories defined on the category of bornological coarse spaces.  We then tried to understand the  coarse analytic  assembly map as a transformation between
coarse homology theories.  We observed   that {the} solution of this task would imply the coarse Baum-Connes conjecture 
in the case of finite asymptotic dimension  as a formal consequence of comparison theorems \cite[Thm.~6.115]{buen}. More generally, 
given a general coarse homology theory one can define the  coarsification of an associated locally finite homology theory and ask for an {analogue} of the coarse analytic assembly map. In this case we would obtain   an analogue of the coarse Baum--Connes conjecture for the coarse homology in question.

As we shall observe in the present paper the successful solution of these problems involves modifying the domain of the coarse assembly map.  Instead of coarsifying locally finite homology theories we coarsify local homology theories.  The first goal
of the present paper is to introduce this notion and to develop a motivic picture of the coarsification process. 
We then define the assembly map using the forget-control map of  the cone sequence.  For more details we refer to Subsection \ref{regoiwjergfrewfwerfwerfwref}. 
The second goal  of the paper explained    in greater detail in Subsection \ref{efioerfefvefvfsv} is to provide various conditions implying that the coarse assembly map is an equivalence. 
Furthermore, under suitable assumptions, we can calculate  the domain of the coarse assembly map in terms of locally finite homology.  
These calculations show that the construction of the present paper solves the original problem in many cases. 
In a separate paper  {\cite{compass}}
we consider the case of coarse $K$-homology theory. We  will show that for nice spaces  the new assembly map is equivalent to the classical  coarse analytic assembly map.

%
%
%
%

\subsection{Construction of the coarse assembly maps}\label{regoiwjergfrewfwerfwerfwref}
 
In the following we describe the set-up in which we will construct the coarse assembly map. The basic category is the  category $\BC$ of bornological coarse spaces introduced
in \cite{buen}.  Let $\bC$  be a   cocomplete, stable $\infty$-category, e.g., the $\infty$-category of spectra $\Sp$.  A $\bC$-valued coarse homology theory is a functor
$$E \colon \BC\to \bC$$  which is coarsely invariant, coarsely excisive, $u$-continuous, and vanishes on flasques. We refer to \cite{buen} for a detailed description of these properties. In order to study properties of coarse homology theories in general we  constructed in \cite{buen} a universal coarse homology theory
$$\Yo^{s} \colon \BC\to \Sp\cX$$
with values in the stable $\infty$-category of motivic coarse spectra $\Sp\cX$. A $\bC$-valued coarse homology   theory  as above is then equivalently described as a colimit preserving functor
\[E\colon \Sp\cX\to \bC\, .
\]

Locally finite homology theories are defined on the category $\TopBorn$ of topological bornological spaces and proper continuous maps \cite[Sec.~7.1]{buen}.  The target  $\bC$ of  a locally finite homology theory will  in addition  be assumed to be  complete. 
A   functor
$$H^{\lf} \colon \TopBorn\to \bC$$ 
is a locally  finite homology theory 
 if, in addition to the  usual homological conditions of excision and homotopy invariance,   it satisfies the local finiteness condition. It requires   that the natural map
 \begin{equation}\label{f34hio34t34g3g3g3gg}
H^{\lf}(X)\to \lim_{B} \Cofib(H^{\lf}(X\setminus B)\to H^{\lf}(X))
\end{equation}
is an equivalence for every bornological topological space $X$,
where the limit runs over the bounded subsets of $X$. Every homology theory $H$ has a corresponding locally finite version $H^{\lf}$ \cite[Def.~7.15 and Prop.~7.37]{buen}.

A particular class of coarse homology theories are  coarsifications $QH^{\lf}$ of   locally finite homology theories  $H^{\lf}$ \cite[Def.~7.44 and Prop. 7.46]{buen}.
In contrast to general coarse homology theories, coarsifications of locally finite homology theories seem to be much more tractable because they can be studied by well-established methods of homotopy theory.

Analytic $K$-homology $K^{\an,\lf}\colon \TopBorn\to \BC$ is an example of a locally finite homology theory \cite[Sec.~7.66]{buen}. For a bornological coarse space $X$ of locally bounded geometry  in \cite[Def.~8.139]{buen} {we constructed}  a version of the coarse analytic assembly map  
\[
\mu\colon QK^{\an,\lf}(X)\to \KX(X)\, ,
\]
where $\KX\colon \BC\to \Sp$ is the coarse $K$-homology.
  We were not able to construct such a morphism for arbitrary spaces $X$. In particular, we do not have a natural transformation between coarse homology theories.

The construction of $\mu$ uses specialities of topological $K$-theory. One could ask whether there are other  pairs of a coarse homology theory $E$ and a locally finite homology theory $H^{\lf}$ which are related by a version of a coarse assembly map $QH^{\lf}(X)\to E(X)$.  
In the present paper we will see that, if interpreted in the appropriate sense, such pairs exist in abundance. One of the difficulties seems to be the fact that 
locally finite homology theories are characterized by a limit condition \eqref{f34hio34t34g3g3g3gg}. It is therefore complicated to construct maps out of locally finite homology theories.
The main novelty of the present paper is to introduce the notion of a local homology theory, essentially by replacing the condition of being locally finite by the weaker condition of vanishing on flasques, {see Definition~\ref{ioojwfefewfwfw}.} 

In the following we explain this in greater detail. We introduce the  category of uniform bornological coarse spaces $\UBC$. A    local  homology theory is then a functor
\[
F\colon\UBC\to \bC\]
which is homotopy invariant, excisive, $u$-continuous,  and vanishes on flasques. 
We will   construct a universal local homology theory
\[
\Yo\cB^{s} \colon\UBC\to \Sp\cB
\]
with values in motivic uniform bornological coarse spectra (Corollary \ref{giugoeggergg}). 
In fact, we will consider two versions of local homology theories distinguished by the condition that the descent axiom
involves open or closed decompositions. The open version will be indicated by adding a superscript ``$o$''.
Similarly as in the case of coarse homology theories, a $\bC$-valued local homology theory  is  equivalently  described as a colimit preserving functor
$$F\colon\Sp\cB\to \bC\, .$$
The nature of the local finiteness condition  \eqref{f34hio34t34g3g3g3gg}  makes it impossible to construct a universal locally finite homology theory in a similar manner.

Any locally finite homology   theory $H^{\lf}$ gives rise to a local homology theory which in the notation of the present paper  appears as $H^{\lf}\circ F_{\cC,\cU/2}$ {in Lemma~\ref{fiweofwefewfewfewf}.}
  
A uniform bornological coarse space has an underlying bornological coarse space.  But if we simply  forget the uniform structure, then we completely lose the local topological structure of the space. A more interesting transition from uniform bornological coarse spaces to bornological coarse spaces keeping the local structure is given by the cone construction. Indeed, with the help of the cone one can encode the uniform structure into a suitable coarse structure.

The cone construction will be investigated in various versions in Section \ref{blijeobereggreerg}; the main version is the one in Definition~\ref{rgfporgergergereg1}, \cite[Ex.~5.16]{buen} and \cite[Def.~9.24]{equicoarse}. It  provides a functor
\[
\cO\colon\UBC\to \BC\, .
\]
The cone and the germs at infinity
\[
\cO^{\infty}\colon\UBC\to \Sp\cX
\]
of the cone (see Definition \ref{eroiergegggg}, \cite[Sec.~5.2.3]{buen} and \cite[Sec.~9.5]{equicoarse}) can  be used to pull-back coarse homology theories to functors defined on $\UBC$. 

Let $E\colon\BC\to \bC $ be a coarse homology theory. It is strong \cite[Def.~4.19]{equicoarse} if it annihilates not only flasque but also weakly flasque bornological coarse spaces.
We interpret $E$ as a colimit preserving functor on $\Sp\cX$ and consider the composition
$$E\cO^{\infty}\colon \UBC\stackrel{\cO^{\infty}}{\to}\Sp\cX\stackrel{E}{\to} \bC\, .$$
\begin{lem}[Lemma \ref{fewoiiofwefewf}]\label{qwierfofeqwdqwedqwede}
If $E$ is strong, then $E \cO^{\infty}$ 
 is a    local homology theory.
\end{lem}

The idea to use some version of cones in order to pull-back coarse homology theories has some history. We refer to Higson--Pederson--Roe
\cite[Prop.~12.1]{higson_pedersen_roe} (coarse $K$-homology),  
Mitchener \cite[Thm.~4.9]{mit} ({coarsely excisive} theories), Bartels--Farrell--Jones--Reich \cite[Sec.~5]{MR2030590} (equivariant coarse algebraic {$K$-homology}), \cite{hamped} (coarse topological and algebraic $K$-theory and $L$-theory), or Weiss
\cite{MR1880196} (algebraic $K$-theory of additive categories and retractive spaces) as entry points to the literature.

Given an entourage $U$ of a bornological coarse space $X$ we can form the Rips complex $P_{U}(X)$ at scale $U$, see Example \ref{efwoiuweofwefewfewfw}. It is a simplicial complex which will be equipped  with the path quasi metric induced by the spherical   metric on its simplices. The metric induces a coarse and a uniform structure on $P_{U}(X)$, and the family of subsets $(P_{U}(B))_{B\in \cB}$ (where $\cB$ denotes the bornology of $X$)  generates the bornology of the uniform bornological coarse space $P_{U}(X)$. There is a canonical embedding  $X\hookrightarrow  P_{U}(X)$ of $X$ into the zero skeleton of $P_{U}(X)$ which induces an equivalence of bornological coarse spaces $X_{U}\to F_{\cU}(P_{U}(X))$, where $F_{\cU}$ forgets the uniform structure.

On the one hand the family   $(F_{\cU}(P_{U}(X)))_{U\in \cC}$ (where $\cC$ denotes the coarse structure of~$X$) of  underlying bornological coarse spaces of the Rips complexes is coarsely equivalent to the constant family on $X$.
One the other hand, forming the colimit of the motivic uniform bornological coarse spectra represented by the Rips complexes, we obtain a functor
\[
\bP\colon\BC\to \Sp\cB
\]
called the universal coarsification, see Definition \ref{ergioerggrg546546}.
In detail,
$$\bP(X)\simeq \colim_{U\in \cC} \Yo^{s}\cB(P_{U}(X))\, .$$
By Proposition \ref{foiwfewfewewfewwe} the functor $\bP$ is a $\Sp\cB$-valued coarse homology theory and can therefore be interpreted as a colimit preserving functor 
\begin{equation}\label{erlkgnmekvleferr}
\bP\colon\Sp\cX\to \Sp\cB\, .
\end{equation}
Pull-back along $\bP$ associates to every  $\bC$-valued local homology theory $F$ a  $\bC$-valued coarse homology theory $F\bP$, see Definition \ref{oijoiergerg}.


For a locally finite homology theory $H^{\lf}$ the coarsification of the local homology theory $H^{\lf}\circ F_{\cC,\cU/2}$ induced from $H^{\lf}$ coincides with the coarsification $QH^{\lf}$ from \cite[Def.~7.44]{buen} which we have discussed earlier, i.e., we have an equivalence
\[
QH^{\lf}\simeq (H^{\lf}\circ \cF_{\cC,\cU/2})\bP\, .
\]

We now state the main construction of this paper. Let $E \colon G\BC\to \bC$ be a strong  coarse homology theory. The cone construction gives rise to a fibre sequence 
\begin{equation}\label{fspoafkapdsfasfasfsdf}
E\circ F_{\cU} \to E\cO \to E\cO^{\infty}\xrightarrow{\partial} \Sigma E\circ F_{\cU}
\end{equation} 
of  local homology theories $\UBC\to \bC$.
For the following definition we interpret \eqref{fspoafkapdsfasfasfsdf} as a fibre sequence of colimit preserving functors $\Sp\cB\to \bC$.
 \begin{ddd}[Definition \ref{fiwjfofewfewfwefefweffw}]\label{fiwjfoe8765rteffw}
The coarse assembly map for~$E$ is the natural transformation between coarse homology theories
\[
\mu_{E} \colon E\cO^{\infty}\bP\to \Sigma E\, ,
\]
derived from the boundary map $\partial$ of the cone sequence \eqref{fspoafkapdsfasfasfsdf} by
precomposition with $\bP$ and using the identification $\Sigma E\simeq \Sigma E\circ \bF\circ \bP$ (see Proposition \ref{frioorfwefewfefwefew}).
\end{ddd}

\subsection{{Isomorphism results and computations}}\label{efioerfefvefvfsv}
 
In Section \ref{ergop34t34t34t34} we study various conditions on the strong coarse homology theory $E$ and the bornological coarse space $X$ which imply that the coarse assembly map
\[
\mu_{X,E} \colon E\cO^{\infty}\bP(X)\to \Sigma E(X)
\]
is an equivalence. 
Let us mention three results which are analogues of  instances of the coarse Baum--Connes and the coarse Farrell--Jones conjecture. We refer to Section~\ref{ergop34t34t34t34} for a detailed description of the assumptions occuring in the following theorems.

Let $X$ be a bornological coarse space and let  $E \colon \BC\to \bC$ be a  strong coarse homology theory.
\begin{theorem}[Theorem~\ref{woifowfwewfewf}]\label{thm54terw234}
If $X$ admits  a cofinal set of entourages $U$ such that $X_{U}$ has finite asymptotic dimension, then
the coarse assembly map $  \mu_{E,X} $ is an equivalence.
\end{theorem}

\begin{theorem}[Theorem~\ref{feoijofwefewfewf}]
Assume:
\begin{enumerate}
\item $\bC$ is compactly generated.
\item $E$ is weakly additive.
\item $E$ admits transfers. 
\item $X$ admits  a cofinal set of entourages $U$ such that $X_{U}$ has finite decomposition complexity.
\end{enumerate}
Then the coarse assembly map $\mu_{E,X}$ is an equivalence.
\end{theorem}

Let $K$ be a simplicial complex, and  let $ K_{d}$ be the corresponding uniform bornological coarse space  whose structures are induced from the path quasi-metric induced by the spherical metric on the simplices. We say that $K_{d}$ is obtained from $K$ by equipping this complex with the metric structures. The symbol 
  $F_{\cU}(K_{d})$ denotes the underlying bornological coarse space of $K_{d}$.
\begin{theorem}[Corollary~\ref{wrfiowefwewfewfw}]\label{efiwjfowefewfewfewfewf}
Assume:
\begin{enumerate}
\item $\bC$ is complete.
\item $E$ is 
additive and admits transfers.
 \item $K$ has bounded geometry.
\item $K_{d}$ is equicontinuously contractible.
\item $K_{d}$ admits a coarse scaling.
\end{enumerate}
Then the coarse assembly map 
$\mu_{E,F_{\cU}(K_{d})}$ is an equivalence.
\end{theorem}
 
Let $E\colon \BC \to \bC$  be a strong
coarse homology theory. 
Then $E\cO^{\infty}$ is a local homology theory, but in general it seems to be difficult to understand its values.
 Fortunately, if $E$ is in addition additive, then on nice spaces it behaves like a locally finite homology theory. Concretely, we have the following result.

Let $X$ be a uniform bornological coarse space.
\begin{prop}[Proposition~\ref{fifowefweewfwef}]
Assume:
\begin{enumerate}
\item $\bC$ is complete.
\item $E$ is   additive.
\item $X$ is homotopy equivalent in $\UBC$ to a countable, locally finite, finite-dimensional simplicial complex equipped with the metric structures.
\end{enumerate}
Then we have a natural equivalence
\begin{equation}\label{doij23oid23d23dd23d23d23d2d}
(\Sigma E(*)\wedge \Sigma^{\infty}_{+})^{\lf}( X)\simeq
E\cO^{\infty}(X)\, .
\end{equation}
\end{prop}
The left-hand side of \eqref{doij23oid23d23dd23d23d23d2d} is the value on the underlying topological  bornological space of $X$ of the locally finite version of the  homology represented by the object $\Sigma E(*)$, see Definition \ref{fwjeofiewffewwfewfw}.

 
The next proposition is a consequence of Proposition  \ref{fifowefweewfwef} applied to Rips complexes. It provides, under appropriate conditions, a calculation of the domain of the coarse assembly map.

Let $E$ be a strong coarse homology theory and $X$ be a bornological coarse space.
\begin{prop}[Proposition \ref{rgriugeroigregregerg}]\label{ewoijofwefewfewf}
Assume:
\begin{enumerate}
\item $\bC$ is complete.
\item $E$ is countably additive.
\item $X$ has  bounded geometry.
\end{enumerate}
Then  we have a natural equivalence
$$((\Sigma E(*)\wedge \Sigma^{\infty}_{+})^{\lf}\circ F_{\cC,\cU/2})\bP^{o}(X)\simeq E\cO^{\infty}\bP(X)\, .$$
\end{prop}
Note that $\bP^{o}$ denotes the open version of the universal coarsification.

Assume that $E\to E^{\prime}$ is a natural transformation of strong coarse homology theories such that $E(*)\to E^{\prime}(*)$ is an equivalence.
Then we can use Proposition \ref{ewoijofwefewfewf} to show for a bornological coarse space $X$ that
$E(X)\to E ^{\prime}(X)$ is an equivalence if the assembly maps
$\mu_{E,X}$ and $\mu_{E^{\prime},X}$ are equivalences, see Theorem \ref{fewoihjowefwfewfwef}.
The precise statement is the following:  
\begin{theorem}[Theorem~\ref{fewoihjowefwfewfwef}]\label{fewo09876fwef1}
Assume:  
\begin{enumerate}
\item $\bC$ is complete.
\item $E$ and $E^{\prime}$ are 
additive.
\item $E(*)\to E^{\prime}(*)$ is an equivalence.
\item $X$ has bounded geometry. 
\item\label{oifjowefewfefewfewfw9} The assembly maps $\mu_{E,X}$ and $\mu_{E^{\prime},X}$ are equivalences.
\end{enumerate}
Then $E(X)\to E^{\prime}(X)$ is an equivalence.
\end{theorem}

It is tempting to apply Theorem \ref{fewo09876fwef1} to the transformation 
$\mu_{E}\colon E\cO^{\infty}\bP\to \Sigma E$ in order to show that $\mu_{E,X}$ is an equivalence.
But in view of Assumption \ref{fewo09876fwef1}.\ref{oifjowefewfefewfewfw9} this would lead to a circular argument.

\paragraph{Acknowledgements}
\textit{The authors were supported by the SFB 1085 ``Higher Invariants'' funded by the Deutsche Forschungsgemeinschaft DFG. The second named author was also supported by the Research Fellowship EN 1163/1-1 ``Mapping Analysis to Homology'' of the Deutsche Forschungsgemeinschaft DFG.}

\textit{The first named author profited from many critical remarks by Denis-Charles Cisinski, Markus Land and the participants of the course on ``Coarse Geometry'' held in Regensburg in the years 2016/17, where parts of this material were first presented.}

\section{Uniform bornological coarse spaces}\label{eoijiqoegergegwegwerg}

In this section we   introduce the category of   uniform bornological coarse spaces and then discuss some basic {constructions} and examples.

Let $X$ be a set. A bornology on $X$ is a subset $\cB$ of the power set $\cP_{X}$
which is closed under forming finite unions, taking subsets, and which contains all one-point sets. 
The elements of $\cB$ are called the bounded subsets of $X$. 
A map between sets equipped with bornologies is said to be proper if preimages of bounded subsets are bounded.

A coarse structure on $X$ is a subset $\cC$ of $\cP_{X\times X}$ which contains the diagonal, is closed under forming  
finite unions and compositions (in the sense of correspondences), the symmetry, and taking subsets.  The elements of $\cC$ are called coarse entourages of $X$.  
A map between sets equipped with coarse structures is said to be controlled if it sends  coarse entourages to  coarse entourages.

A  bornology $\cB$ and a coarse structure $\cC$ on the same set are said to be compatible   if $\cB$ is stable under forming thickenings by coarse entourages in $\cC$ (see \cite[Def.~2.6]{buen}).

A set equipped with compatible bornological and coarse structures is an object of the category $\BC$ of  bornological coarse spaces.  A morphism between bornological coarse spaces 
is a proper and controlled map. We refer to \cite[Ch.~2]{buen} for a  detailed study of the category $\BC$.

Finally, {a} uniform structure on a set $X$ is   a subset $\cU$ of $\cP_{X\times X}$ which 
 is closed under forming  
finite intersections, compositions, the symmetry, and taking supersets. In addition we require that every element $U$ of $\cU$ contains the diagonal and admits a subset $V$ in $\cU$ such that $V\circ V\subseteq U$ ($V\circ V$ denotes the composition of $V$ with itself). The elements of $\cU$ are called uniform entourages.  A uniform space is a set equipped with a uniform structure.  
A map between the underlying sets of uniform spaces is called uniformly continuous if
preimages of uniform entourages are uniform entourages. 

A uniform structure $\cU$ and a coarse structure $\cC$ on the same set are said to be compatible  if $\cU\cap \cC\not=\emptyset$  (see \cite[Def.~5.4]{buen}). For more details we refer to \cite[Ch.~5.1]{buen} and \cite[Sec.~9.1]{equicoarse}.

A bornological coarse space with an additional compatible uniform structure
is called a uniform bornological coarse space.
We let $\UBC$ denote the category of   uniform bornological coarse spaces
and proper, controlled and uniformly continuous maps. 

A bornology   $\cB$ and a topology $\cS$ on the same set $X$ 
are called compatible if  $\cB\cap \cS$ is cofinal in $\cB$, and $\cB$ is closed under forming closures (\cite[Def.~7.1]{buen}). 
A set with a  bornology   $\cB$ and a compatible topology $\cS$ is called a topological bornological space. We  let $\BT$ be the category of topological bornological spaces and continuous and proper maps, see \cite[Sec.~7.1.1]{buen}.

 A uniform structure $\cU$ on a set naturally induces a topology generated by the $U$-thickenings $U[\{x\}]$ of the one-points sets  for all $U$ in $\cU$ and $x$ in $X$.
 We have a forgetful functor
$$F_{\cC,\cU/2}\colon \UBC\to \BT$$
which forgets the coarse structure and only  remembers the bornology and the topology induced from the uniform structure.
\color{black}




Let $(X,\cU)$ be a uniform space.
\begin{ddd}
The coarse structure associated to the uniform structure is defined by
$$\cC(\cU):=\bigcap_{V\in \cU} \cC\langle  \{V \}\rangle\, .$$
Here $ \cC\langle  \{V \}\rangle$ denotes the coarse structure generated by $V$ (\cite[Ex.~2.12]{buen}).
\end{ddd}

\begin{ex}
The coarse structure $\cC(\cU)$ is not necessarily compatible with the uniform structure $\cU$, as the following example shows. We let $X := \{0\} \cup \{1/n \:|\: n \in \IN\}$. We further define  the uniform structure $\cU$ of $X$  to be the one induced from the metric on $X$ induced from the canonical inclusion $X \subseteq \IR$. This uniform structure is generated by the uniform entourages $U_r$ for all $r > 0$ given by $U_r := \{(x,y) \in X \:|\: d(x,y) < r\}$. We now observe that  $\cC(\cU) $ is the minimal coarse structure consisting of all subsets of the diagonal. It is not compatible with $\cU$.
\hB
\end{ex}

\begin{ex} 
The notion of a quasi-metric on a set is defined similary as the notion of a metric where one in addition allows that points have infinite distance.
For example, a disjoint union of metric spaces is naturally a quasi-metric space. The definition of a coarse structure associated to a metric \cite[Ex.~2.18]{buen} generalizes immediately to the case of quasi-metric spaces.  Similarly, a quasi-metric also induces a uniform structure.

We consider a quasi-metric space with the induced coarse and uniform structures $\cC$ and $\cU$. They are compatible. 
If the space is in addition  a path quasi-metric space (i.e. each path component is a path metric space, and different path components  have infinite distance {to each other}), then we have the equality $\cC=\cC(\cU)$. In particular, in this case $\cC(\cU)$ is compatible with $\cU$.
\hB
\end{ex}


\begin{rem}
A map between metric spaces $f:(X,d)\to (X^{\prime},d^{\prime})$ is called uniformly continuous if for every $\delta$ in $(0,\infty)$ there exists an $\epsilon$ in $(0,\infty)$ such that for all pairs of points $x,y$ of $X$ with $d(x,y)\le\epsilon$ we have $d(f(x),f(y))\le \delta$.
 A uniformly continuous map between metric spaces in this sense
 is uniformly continuous as a map between the associated uniform spaces. \hB
\end{rem}

\begin{ex}\label{ifowfewefewfw}
Let $X$ be a simplicial complex. Then $X$ has a canonical spherical path quasi metric which induces a coarse structure $\cC$ and a compatible uniform structure $\cU$. 

A choice of a set $A$ of sub-complexes
generates a bornology $\cB \coloneqq \cB\langle A\rangle$. It is compatible with the coarse structure if for every entourage $U$ in $\cC$ and every sub-complex $K$ in $A$ there exists another sub-complex $K^\prime$ in $A$ with $U[K] \subseteq K^\prime$. The triple $(X,\cC,\cB,\cU)$ is a uniform bornological coarse space. 

If $X^{\prime}$ is a second simplicial complex with a choice of a set 
$A^{\prime}$ of sub-complexes and  $f\colon X\to X^{\prime}$ is a simplicial map such that for every $Y^{\prime}$  in $A^{\prime}$ we have $f^{-1}(Y^{\prime})\in \cB\langle A\rangle$, then
$f$ is a morphism of uniform bornological coarse spaces.
\hB
\end{ex}

\begin{ex}\label{efwoiuweofwefewfewfw}
 If $X$ is a bornological coarse space  and $U$ is a coarse entourage of $X$, then we consider the simplicial complex
$P_{U}(X)$ of probability measures on $X$ which have finite, $U$-bounded support. For a subset $Y$ of $X$ we let $P_{U}(Y)$ denote the sub-complex of $P_{U}(X)$ of measures supported on $Y$.
We  let $A$ be the set of sub-complexes  $P_{U }(B)$  for all bounded subsets $B$ of $X$. The constructions explained in Example \ref{ifowfewefewfw} turn   $P_{U}(X)$ into a uniform bornological coarse space.

Let $f \colon X\to X^{\prime}$ be a morphism between bornological coarse spaces. Then there exists a coarse entourage    $U^{\prime}$   of $X^{\prime}$ such that $(f\times f)(U)\subseteq U^{\prime}$. The push-forward of measures provides a morphism $P_{U}(X)\to P_{U^{\prime}}(X^{\prime})$  between uniform bornological coarse spaces in a functorial way.
\hB
\end{ex}

\begin{ex}
Let $X$ be a uniform bornological coarse space. If $Y$ is a subset of $X$, then $Y$ has an induced  uniform bornological coarse structure. If not said differently, we will always consider subsets with the induced structures. The inclusion $Y\to X$ is then a morphism between uniform bornological coarse spaces.
\hB
\end{ex}

\section{Local homology theories}

In this section we introduce the notion of a local homology theory.
We will actually consider two variants which are distinguished by the details of the excision axiom which
involves closed or open decompositions.

\begin{rem}\label{qrgoijherwiogwerg}
The reason for considering two variants is that  we want to apply the theory to two examples with different properties.
On the one hand, analytic $K$-homology naturally satisfies excision for closed decompositions
\cite[Prop.~7.62]{buen} (applied to $K^{\an,\lf}$).
On the other hand, the suspension spectrum functor $\Sigma^{\infty}_{+}$ naturally satisfies open excision \cite[Ex.~7.40]{buen}. 
\hB
\end{rem}

Let $(X,\cU)$ be a uniform space, and let $A$ and $B$ be subsets  of $X$ with $A\cup B=X$. 

\begin{ddd}
$(A,B)$ is called a closed ({or} open) decomposition, if $A$ and $B$ are closed (resp.~open).
\end{ddd}
For an entourage $U$ let $\cP_{X\times X}^{\subseteq U}$ denote the set of elements of $\cP_{X\times X}$ (the power set of $X\times X$) which are  contained in $U$.
The following is taken from \cite[Def.~{5.19}]{buen}:
\begin{ddd}\label{ifpopwefwfs}
The pair $(A,B)$  is a uniformly excisive decomposition of $X$ if
there exists a uniform entourage $U$ and a function
$\kappa \colon \cP_{X\times X}^{\subseteq U} \to \cP_{X\times X}$ such that:
\begin{enumerate}
\item\label{fuwehfiwuefwfe} The restriction of  $\kappa$ to $\cU\cap \cP_{X\times X}^{\subseteq U}$ is $\cU$-admissible (see Remark \ref{wthgwrfrefs}).
\item For every
$W$ in $\cP_{X\times X}^{\subseteq U}$ we have
$W[A]\cap W[B]\subseteq \kappa(W)(A\cap B)$.
\end{enumerate}
\end{ddd}

\begin{rem}\label{wthgwrfrefs}
Note that in Definition \ref{ifpopwefwfs}  we consider $\cP_{X\times X}^{\subseteq U}$ and $\cP_{X\times X}$ as partially ordered sets with the order relation given by the opposite of the inclusion relation.  By definition, a function between partially ordered sets is order preserving.

Condition~\ref{ifpopwefwfs}.\ref{fuwehfiwuefwfe} means that for every $V$ in $\cU$ there exists $W$ in $ \cU\cap \cP_{X\times X}^{\subseteq U}$ such that $\kappa(W)\subseteq V$.
\hB
\end{rem}

For a coarse space $(X,\cC)$ the notion of a coarsely excisive decomposition \cite[Def.~3.40]{buen} is defined similarly. We again consider two subsets $A$ and $B$ of $X$ such that $X=A\cup B$.

\begin{ddd}
The pair $(A,B)$  is a   coarsely excisive decomposition of the space $X$ if for every 
$V$ in $\cC$ 
there exists a 
$W$  in $\cC$ 
such that we have
$V[A]\cap V[B]\subseteq W(A\cap B)$.
\end{ddd}

%
%
%

%

%

%
%

Let $(X,\cU)$ be a uniform space.
\begin{lem}\label{foiwefefewfewf}
{If $\cC(\cU)$ is compatible with $\cU$, then any} uniformly excisive  decomposition $(A,B)$ of $X$ is coarsely excisive  for  the coarse structure $\cC(\cU)$.
\end{lem}

\begin{proof}
Let $U$ and $\kappa$ be as in Definition \ref{ifpopwefwfs}. Since $\cC(\cU)$ is compatible with $\cU$, after replacing $U$ by a smaller uniform entourage if necessary, we can assume that $U$ is also a coarse entourage, and that $\kappa(W)$ is a coarse entourage for every $W$ in $\cU\cap \cP_{X\times X}^{\subseteq U}$.

Let $V$ be an entourage in $\cC(\cU)$.
Then there exists  an integer $n$ such that $V\subseteq U^{n}$.
We claim that
\[
V[A]\cap V[B]\subseteq  (U^{2n+1}\circ \kappa(U))(A\cap B)\,.
\]

Let $z$ be a point in $V[A]\cap V[B]$. Then there exists integers $r$ and $s$ with $r\le n$ and $s\le n$ and 
 a   sequence of points $(x_{0},\dots,x_{r+s})$
   in $X$ such that $x_{0}\in A$, $x_{r+s} \in B$,  and
 $(x_{i},x_{i+1})\in U$ for all  $i=0,\dots,r+s-1$. 
 There exists $i_{0}$ in $ \{0,\dots,r+s-1\}$ such that
 $x_{i_{0}}\in A$ and $x_{i_{0}+1}\in B$. But then
 $x_{i_{0}+1}\in U[x_{i}]$, i.e., $x_{i_{0}+1}\in U[A]\cap U[B]$. Since $(A,B)$ is uniformly excisive by assumption there
exists a point $y$ in $A\cap B$ such that
$(x_{i_{0}+1},y)\in \kappa(U)$. This now implies that
$z\in (U^{2n+1}\circ \kappa(U))(A\cap B)$ as asserted.
\end{proof}
  
\begin{ex}\label{wfiofwefewfwefewf}
On a path quasi-metric space every closed decomposition is coarsely and uniformly excisive \cite[Ex.~5.21]{buen}.
\hB
\end{ex}

Let $\bC$ be a cocomplete stable $\infty$-category and consider a functor
  $E \colon \UBC\to \bC$.
\begin{ddd}\label{reffoiweffewff}
  We say that $E$ satisfies closed (open) excision if $E(\emptyset)\simeq 0$ and for every   \ubs  $X$    and uniformly  and coarsely excisive closed (open) decomposition  $(A,B)$ of $X$  the square
  $$\xymatrix{E(A\cap B)\ar[r]\ar[d]&E(A)\ar[d]\\E(B)\ar[r]&E(X)}$$
  is cocartesian.
\end{ddd}


The categories $\BC$ and $\UBC$ have symmetric monoidal structures denoted by $\otimes$ (see  \cite[Ex.~2.32]{buen} for the latter) such that the forgetful functor
 $\UBC\to \BC$  has a symmetric monoidal refinement and the underlying uniform space of a tensor product is a cartesian product of the underlying uniform space of the factors. 


 The unit interval $[0,1]$ has a canonical uniform bornological coarse structure (the maximal coarse and bornological structure, and the metric uniform structure). 
 The product
 $[0,1]\otimes X$ is now defined, and the projection $[0,1]\otimes X\to X$ is a morphism of uniform bornological coarse spaces since $[0,1]$ is bounded.

Let $E \colon \UBC\to \bC$ be a functor.
\begin{ddd}
 We say that $E$ is homotopy invariant if for every uniform bornological coarse space $X$ the morphism  $E([0,1]\otimes X)\to E(X)$ induced by the projection is an equivalence.
 \end{ddd}
 
A homotopy between morphisms $f_{0},f_{1}:X\to Y$ of uniform bornological coarse spaces is a morphism $h:[0,1]\otimes X\to Y$ which restricts to $f_{i}$ at the endpoints of the interval.
If $E$ is homotopy invariant, then for homotopic morphisms $f_{0},f_{1}$ we have $E(f_{0})\simeq E(f_{1})$.

A uniform bornological coarse space $X$ is called flasque with flasqueness implemented by  a morphism $f\colon X\to X$ if $f$  implements flasqueness in the sense of bornological coarse spaces \cite[Def.~3.21]{buen} and $f$ is in addition 
 homotopic to the identity.

Let $E\colon \UBC\to \bC$ be a functor.
\begin{ddd}
We say that $E$  vanishes on flasques if $E(X)\simeq 0$ for every flasque uniform bornological coarse space $X$.
\end{ddd}

Let $X$ be a uniform bornological coarse space and $U$ be an entourage  which is both coarse and uniform. Then for every
coarse entourage $V$ such that $U\subseteq V$ we can replace the coarse structure by the coarse structure generated by $V$ and obtain a uniform bornological coarse space $X_{V}$. 
Hence, for a uniform bornological coarse space $X$ the \ubs $X_{V}$ is well-defined for sufficiently large coarse entourages $V$. We have a canonical morphism $X_{V}\to X$ given by the identity of the underlying sets. Hence
the colimit and the canonical morphism in the following definition have a well-defined interpretation. 

Let $E\colon \UBC\to \bC$ be a functor.
\begin{ddd}
We say that $E$ is $u$-continuous if for every \ubs $X$ the canonical morphism
$\colim_{V\in \cC} E(X_{V})\to E(X)$ is an equivalence.
\end{ddd}

Let $\bC$ be a cocomplete stable $\infty$-category, and let $E:\UBC\to \bC$ be a functor.
\begin{ddd}\label{ioojwfefewfwfw}
$E$ is called a closed (open) 
local homology theory if the following conditions are satisfied:
\begin{enumerate}
\item $E$ satisfies closed (open) excision.
\item $E$ is homotopy invariant.
\item $E$ vanishes on flasques.  
\item $E$ is $u$-continuous.
\end{enumerate}
\end{ddd}

We have a forgetful functor
\begin{equation}
\label{eqkj4zrtet}
F_{\cU} \colon \UBC\to \BC
\end{equation}
which forgets the uniform structure.

Let $E:\BC\to \bC$ be a functor.
\begin{lem}\label{fiowefewfewfwe}
If $E $ is a   coarse   homology theory, then $E\circ F_{\cU} \colon \UBC\to \bC$ is a closed and open local homology theory.
\end{lem}

\begin{proof}
It follows from the definition of a coarse homology theory \cite[Def.~4.22]{buen}
 that
 $E\circ F_{\cU}$  is homotopy invariant, $u$-continuous, and vanishes on flasques.
 The functor $F_{\cU}$ sends uniformly and  coarsely excisive  closed or open decompositions  to coarsely excisive ones. Hence by \cite[Cor.~4.28]{buen} the composition
 $E\circ F_{\cU}$ is excisive.  
\end{proof}

\begin{rem}
The reason that the proof of the Lemma \ref{fiowefewfewfwe} is not completely trivial is that excision for coarse homology theories was not defined in terms of coarsely excisive  decompositions but with complementary pairs, see \cite[Def. 4.22.1]{buen}. Our main reason for doing this was that the intersection of
a  coarsely excisive  decomposition with a subset {needs not to be} coarsely excisive, while {the} intersection with subsets preserves complementary pairs. 
\hB
\end{rem}


A locally finite homology theory in the sense of \cite[Def.~7.27]{buen} is a functor
$\BT\to  \bC$ with a complete and cocomplete stable target which is homotopy invariant, locally finite and satisfies weak excision  \cite[Def.~7.26]{buen}. It will
will be called closed or open if it satisfies excision for  closed or open, respectively, decompositions. 
 
\begin{ex}
An example of a closed spectrum-valued locally finite homology theory is the analytic $K$-homology $K^{\an,\lf}$ constructed in \cite[Def.~7.66]{buen}.
 
An example of a locally finite homology theory satisfying open excision is the locally finite version of stable homotopy $\Sigma^{\infty}_{+}(-)^{\lf}$, see \cite[Ex.~7.40]{buen}.
\hB
\end{ex}

 We assume that $\bC$ is a complete and cocomplete stable $\infty$-category.
  Let $E \colon \BT\to \bC$ be a   locally finite  homology theory.   
 \begin{lem}\label{fiweofwefewfewfewf}
 If $E$ is closed (open), then
  $E\circ F_{\cC,\cU/2}$ is a 
  closed (open) local homology theory. 
 \end{lem}

\begin{proof}
 Homotopy invariance    of $E$ implies homotopy invariance of $E\circ F_{\cC,\cU/2}$. {The functor $F_{\cU}$ sends uniformly and  coarsely excisive  closed (open) decompositions to closed (open) decompositions. Since we assume that $E$ is closed (open) the composition $E\circ F_{\cC,\cU/2}$ satisfies closed (open) excision.} Since $F_{\cC,\cU/2}$  forgets the coarse structure, the composition $E\circ F_{\cC,\cU/2}$ is $u$-continuous.

The functor $F_{\cC,\cU/2}$ sends flasque \ubs to topological bornological spaces which are flasque in the sense of   \cite[Def.~7.19]{buen}. Since $E$ vanishes on flasque topological bornological spaces  by \cite[Lem.~7.21]{buen} we conclude that $ E\circ F_{\cC,\cU/2}$ vanishes on flasques.
\end{proof}


%


\section{Motives and the universal local homology theory}

In this section we construct the universal closed  and open  local homology theories. We will write out the details for the closed case. The open case is completely analogous and is obtained by replacing the word ``closed'' by the word ``open'' at the appropriate places.

The construction of the universal local homology theories is completely analogous to the construction of the universal coarse homology theory carried out in \cite[Sec.~3 \& 4]{buen}.  We keep the present section as short as possible and refer to \cite{buen}   for more background and references to the $\infty$-category literature.

\begin{rem}
In order to fix set-theoretic size issues we fix a sequence of three Grothendieck universes whose elements will be called very small sets, small sets, and large sets, respectively.  The  geometric objects in $\BC$, $\TopBorn$ or $\UBC$, and the indexing families of big familes, etc.\ are assumed to belong to the very small universe. The three categories themselves are small. 

For a small $\infty$-category $\bD$ we use the standard notation
\[
\PSh(\bD) \coloneqq \Fun(\bD^{op},\Spc)
\]
for the presentable $\infty$-category {of} space-valued presheaves, where $\Spc$ is the large category of small spaces. The $\infty$-category  $\PSh(\bD)$ is large.  If $\bD$ is an ordinary category, then we consider it as an $\infty$-category using the nerve.  
 
If not said differently, completeness or cocompleteness requires the existence of all limits or colimits indexed by very small index categories, respectively. 
\hB
\end{rem} 



Let $E$ be in $ \PSh (\UBC)$.
\begin{ddd}
$E$ satisfies closed descent if $E(\emptyset)\simeq *$ and for any  \ubs $X$ and uniformly  and coarsely excisive closed decomposition $(A,B)$  of $X$ 
the square 
  $$\xymatrix{E( X)\ar[r]\ar[d]&E(A)\ar[d]\\E(B)\ar[r]&E(A\cap B)}$$
  is cartesian. 
\end{ddd}

 \begin{ddd}\label{qeroigjqegegwegregw}
We let $\Sh (\UBC)$ be the full subcategory of $\PSh (\UBC) $ of presheaves which satisfy closed descent. Its objects will be called sheaves.
\end{ddd}

\begin{lem}
We have a 
 localization
$$L:\PSh (\UBC)\leftrightarrows \Sh (\UBC): \mathit{inclusion}\, .$$
\end{lem}
\begin{proof}
We use the theory developed in  \cite[5.5.4]{htt}. We let $Y\colon \UBC\to \PSh (\UBC)$ denote the Yoneda embedding.
 For every $X$ in $\UBC$ and  uniformly  and coarsely excisive closed decomposition $(A,B)$
of $X$ we consider the morphism
$$Y(A)\sqcup_{Y(A \cap B)}Y(B)\to Y(X)$$ in $\PSh(\UBC)$. Furthermore we consider $\emptyset\to Y(\emptyset)$.
Then by definition  $\Sh (\UBC)$ is the full subcategory of  $\PSh(\UBC)$ of objects which are local with respect to the small set of all such maps. 
The existence of the localization now follows from \cite[5.5.4.15]{htt}.
\end{proof}
 

\begin{lem}\label{fweioewfewfewfewf}
For every   \ubs $X$ the presheaf $Y(X)$ satisfies descent.
\end{lem}

\begin{proof}
See \cite[Lem.~3.12]{buen} for a similar argument.
\end{proof}

Let $E$ be in $ \Sh (\UBC)$.

\begin{ddd}
$E$ is homotopy invariant if for any \ubs  $X$ 
the morphism $E( X)\to  E([0,1]\otimes X)$
induced by 
the projection  is  an equivalence.
  \end{ddd} 

We let $\Sh^{h}(\UBC)$ denote the full subcategory of $\Sh (\UBC)$  of homotopy invariant sheaves.

\begin{lem}
We have a localization 
$$\cH:\Sh (\UBC)\leftrightarrows \Sh^{h}(\UBC):\mathit{inclusion}\, .$$
\end{lem}
\begin{proof}
For every $X$ in $\UBC$ we consider the map $Y([0,1]\otimes X)\to Y(X)$ in $\Sh(\UBC)$ (using Lemma \ref{fweioewfewfewfewf}).
Then  $ \Sh^{h}(\UBC)$ is the full subcategory of $\Sh (\UBC)$ of all objects which are local with respect to the small set of all such maps.   The existence of the localization now follows from \cite[5.5.4.15]{htt}.
\end{proof}

We call $\cH$ the homotopification.
 
Let $E$ be in $ \Sh^{h}(\UBC)$ and let  $*^{h}$ denote a  final object  of $\Sh^{h}(\UBC)$.
\begin{ddd}
We say that $E$ vanishes on flasques if for every flasque  \ubs $X$   we have an equivalence 
$E(X)\simeq *^{h}$.
\end{ddd}
 We let  $\Sh^{h,\mathit{fl}}(\UBC)$ denote the full subcategory of $\Sh^{h} (\UBC)$ of homotopy invariant sheaves which vanish on flasques.

\begin{lem}
We have a localization
$$\Fl:\Sh^{h} (\UBC)\leftrightarrows \Sh^{h,\mathit{fl}}(\UBC):\mathit{inclusion}\, .$$
\end{lem}

\begin{proof}
Let  
 $\emptyset^{h}$  denote an initial object of  $\Sh^{h}(\UBC)$.
For every  flasque $X$ in $\UBC$ we consider the morphism
$\emptyset^{h}\to \cH(Y(X))$ in $\Sh^{h}(\UBC)$. 
  Then by definition $\Sh^{h,\mathit{fl}}(\UBC)$ is the full subcategory of
 $ \Sh^{h} (\UBC)$ of objects which are local with respect to the small set of all such maps.
 The existence of the localization now follows from \cite[5.5.4.15]{htt}.
\end{proof}

For a \ubs  $X$ let $\tilde \cC$ denote the subset of  coarse
entourages  which are also uniform. Note that this subset is cofinal in $\cC$.

Let $E$ be  in {$\Sh^{h,\mathit{fl}}(\UBC)$.} 
\begin{ddd}
We say that $E$ is $u$-continuous if for every uniform bornological coarse space $X$    the natural morphism
$E(X)\to \lim_{U\in \tilde \cC} E(X_{U})$ is an equivalence.
\end{ddd}

 We let $\Sh^{h,\mathit{fl},u}(\UBC)$ denote the full subcategory of $\Sh^{h,fl} (\UBC)$ of homotopy invariant sheaves which vanish on flasques and are $u$-continuous.  

\begin{lem}
We have an adjunction
$$U:\Sh^{h,\mathit{fl}} (\UBC)\leftrightarrows \Sh^{h,\mathit{fl},u}(\UBC):\mathit{inclusion}\, .$$
\end{lem}

\begin{proof}
For every $X$ in $\UBC$ we consider the morphism
$$\colim_{U\in \tilde \cC} \Fl(\cH(Y(X_{U})))\to  \Fl(\cH(Y(X))$$ in
$\Sh^{h,\mathit{fl}} (\UBC)$.
Then by definition $\Sh^{h,\mathit{fl},u}(\UBC)$ is the full subcategory of objects which are local with respect to the small set of all maps as above.  The existence of the localization now follows from \cite[5.5.4.15]{htt}. 
\end{proof}

\begin{kor}
The $\infty$-categories from above, i.e., $\Sh(\UBC)$,  $\Sh^{h }(\UBC)$, $\Sh^{h,\mathit{fl}}(\UBC)$ and $\Sh^{h,\mathit{fl},u}(\UBC)$ are all presentable.
\end{kor}


\begin{ddd}
We call $\Sh^{h,\mathit{fl},u}(\UBC)$  the category of closed motivic  uniform bornological coarse spaces and  use the notation $\Spc\cB$.
\end{ddd}
   

%
%
%

The following process of stabilization  is analogous to the one described in \cite[Sec.~4.1]{buen}.
\begin{ddd}
We define the stable $\infty$-category of closed motivic uniform bornological coarse spectra
$\Sp\cB$     as the stabilization of $\Spc\cB$    in the world of the  presentable stable $\infty$-categories and left-adjoint functors.
\end{ddd}

We have a canonical functor
$$\Sigma^{\infty}_{+}\colon \Spc\cB\to \Sp\cB\ .$$
By construction the category $\Sp\cB$ is a presentable stable $\infty$-category.

We have a functor
$$\Yo\cB^{s} \coloneqq \Sigma^{\infty}_{+}\circ U\circ \Fl\circ \cH\circ L\circ Y \colon \UBC \to \Sp\cB\, .$$ In view of Lemma \ref{fweioewfewfewfewf} we could omit $L$ in this composition.
For a \ubs  $X$  we call 
$\Yo\cB^{s}(X)$ the   closed motive of $X$.

By construction the functor
$\Yo\cB^{s}$ is a $\Sp\cB$-valued closed local homology theory.
It is in fact  the universal closed local homology theory.

\begin{rem}
In order to distinguish the open case (constructed in an analogous manner using open decompositions in Definition \ref{qeroigjqegegwegregw}) from the closed case we will use the notation
$$\Yo\cB^{s,o} \colon \UBC\to \Sp\cB^{o}$$
for the corresponding functor in the open case.
\hB
\end{rem}

\begin{kor}\label{giugoeggergg}
If $\bC$ is a   stable, small cocomplete $\infty$-category, then
precomposition with $\Yo\cB^{s}$ induces an equivalence between the $\infty$-categories of $\bC$-valued closed local homology theories and small colimit preserving functors $\Sp\cB\to \bC$.
\end{kor}

\begin{rem}
If $\bC$ is presentable, then the assertion of Corollary \ref{giugoeggergg} is immediate from the universal properties of the localizations leading to $\Sh^{h,\mathit{fl},u}(\UBC)$ and the process of stabilization. For the general case see  \cite[Lem.~4.4]{buen} and the  discussion  
before.
\hB
\end{rem}

For a closed local homology theory $E$ we use the notation $E$ also to denote the corresponding colimit preserving functor $\Sp\cB\to \cC$.

%
%


\begin{rem}
The existence of non-trivial closed local homology theories (see Lemma~\ref{fiowefewfewfwe} and Lemma~\ref{fiweofwefewfewfewf}) shows that the category $\Sp\cB$ is non-trivial.
\hB
\end{rem}

\section{The universal coarsification functor \texorpdfstring{$\bP$}{P}}
 
 In this section we  extend the construction given in Example \ref{efwoiuweofwefewfewfw} 
to a coarse homology theory $\bP$ called the universal coarsification.

 Let $\BC^{\cC}$ be the category of pairs $(X,U)$ of bornological coarse spaces $X$ and coarse entourages $U$ of $X$. A    morphisms $f\colon (X,U)\to  (X^{\prime},U^{\prime})$ is a  morphism $f \colon X\to X^{\prime}$ of bornological coarse spaces such that $(f\times f)(U)\subseteq U^{\prime}$. By Example \ref{efwoiuweofwefewfewfw}  we have a functor
$$P\colon \BC^{\cC}\to \UBC$$ which sends $(X,U)$ to the \ubs  $P(X,U)$
associated to the simplical complex $P_{U}(X)$ and the family $A$ of sub-complexes $P_{U }(B)$ for all bounded subsets $B$ of $X$. 
We  furthermore have a forgetful functor
\begin{equation}
\label{eqjk3re}
F_{\cC}\colon \BC^{\cC}\to \BC
\end{equation}
which sends the pair $(X,U)$ to $X$.
 
\begin{ddd}\label{ergioerggrg546546}
We define the universal coarsification functor $\bP\colon \BC\to \Sp\cB$   as the left Kan-extension
$$\xymatrix{\BC^{\cC}\ar[rr]^-{\Yo\cB^{s}\circ P}\ar[d]_-{F_{\cC}} & & \Sp\cB\,.\\
\BC\ar[urr]_-{\bP} & &}$$
\end{ddd}


By the following proposition we can interpret $\bP$ also as a   colimit preserving functor
 $$\bP\colon \Sp\cX\to \Sp\cB\, .$$
\begin{prop}\label{foiwfewfewewfewwe}
The universal coarsification functor $\bP$ is a $\Sp\cB$-valued coarse homology theory.
\end{prop}

\begin{proof}
We verify the axioms for a coarse homology theory \cite[Def.~4.22]{buen}.
If $X$ is a bornological coarse space with coarse structure $\cC$, then by the point-wise formula for the left Kan extension
\begin{equation}\label{oirjeorgergregr}
\bP(X)\simeq \colim_{U\in \cC} \Yo\cB^{s}(P(X,U))\, .
\end{equation} 
We have equivalences
\begin{eqnarray*}
\colim_{V\in \cC} \bP(X_{V})&\simeq&
\colim_{V\in \cC} \colim_{U\in \cC\langle V\rangle} \Yo\cB^{s}(P(X,U))   \\&\simeq&
\colim_{U\in \cC} \Yo\cB^{s}(P(X,U))\\&\simeq&\bP(X)\, ,
\end{eqnarray*}
where for the second equivalence we use  a cofinality consideration. Hence 
 $\bP$ is $u$-continuous.
 
 Consider two morphisms $f_{0},f_{1}\colon (X,U)\to (X^{\prime},U^{\prime})$ in $\BC^{\cC}$.
 If $f_{0}$ and $f_{1}$ are $U^{\prime}$-close to each other, i.e., $(f_{0},f_{1})(\diag(X))\subseteq U'$, then
 $P(f_{0})$ and $P(f_{1})$ are homotopic and $\Yo\cB^{s}(P(f_{0}))\simeq \Yo\cB^{s}(P(f_{1}))$ by the homotopy invariance of $\Yo\cB^{s}$. This implies that $\bP$ is coarsely invariant.
 
Note that $\bP(\emptyset)\simeq \Yo\cB^{s}(\emptyset)\simeq 0$.
  Let $(X,U)$ be an object of $ \BC^{\cC}$ such that $U$ contains the diagonal of $X$. For a subset $Y$ of $X$ note that $P_{U }(Y)$ is a closed subset of $P_{U}(X)$.
 If $(Z,\cY)$ with $\cY=(Y_{i})_{i\in I}$ is a complementary pair, then for every  $i,j$ in $I$ such that $Z\cup Y_{i}=X$
 and {$U[Y_{i}]\subseteq Y_{j}$} the pair
 $(P_{U}(Z),P_{U}(Y_{j}))$ is a closed decomposition of the path quasi-metric space $P_{U}(X)$ and hence uniformly  and coarsely excisive (see Example \ref{wfiofwefewfwefewf}).  For sufficiently large $j$ in $I$ and since  $\Yo\cB^{s}$ is excisive in the sense of Definition \ref{reffoiweffewff} for closed uniformly  and coarsely excisive decompositions we get a cocartesian square
 $$\xymatrix{\Yo\cB^{s} (P_{U}(Z)\cap P_{U}(Y_{j})) \ar[r]\ar[d]&\Yo\cB^{s}(P_{U}(Z))\ar[d]\\\Yo\cB^{s}(P_{U}(Y_{j}))\ar[r]&\Yo\cB^{s}(P_{U}(X))}$$
  We  form the colimit over $j$ in $I$  and over $U$ in the coarse structure of $X$. 
  The lower right corner yields
  $\bP(X)$. For the lower left corner we first take the $U$-colimit and then the $j$-colimit. Then we obtain the object
  $\bP(\cY)$. In the upper right corner we get $\bP(Z)$.
  For the upper left corner we note that
  $P_{U}(Z)\cap P_{U}(Y_{j})=P_{U}(Z\cap Y_{j})$ and finally  get
  $\bP(Z\cap \cY)$. Since we have exhibited the square
  $$\xymatrix{\bP( Z \cap \cY) \ar[r]\ar[d]& \bP(Z)\ar[d] \\ \bP(\cY) \ar[r]&\bP (X)}$$
  as a colimit of  cocartesian squares it is cocartesian itself.
    In view of \cite[Rem.~4.23]{buen}
we can conclude that the functor $\bP$ satisfies excision. 

Finally, assume that a bornological coarse space $X$ is flasque with flasqueness implemented by  $f \colon X\to X$.
Let $U$ be an entourage of $X$ such that $\id_{X}$ and $f$ are $U$-close to each other. Then $V \coloneqq \bigcup_{n\in \nat} (f^{n}\times f^{n})(U)$ is again an entourage of $X$ which contains $U$.  Now note that $(f\times f)(V)\subseteq V$. Therefore $P(f):P(X,V)\to P(X,V)$ is defined.  This map implements flasqueness of $P(X,V)$, hence $\Yo\cB^{s}(P(X,V))\simeq 0$.
 In view of \eqref{oirjeorgergregr} by cofinality we see that $\bP$ vanishes on flasques.
\end{proof}

If $E$ be a closed $\bC$-valued local homology theory, then we can  define the  coarse homology theory
$$E\bP \colon \BC\stackrel{\bP}{\to}  \Sp\cB\stackrel{E}{\to} \bC\, .$$
\begin{ddd}\label{oijoiergerg}
The coarse homology theory $E\bP $ will be called the coarsification of $E$.
\end{ddd}

\begin{ex}
  In Example \ref{fiweofwefewfewfewf} we have seen that for any closed locally finite homology theory
  $E\colon \BT\to \bC$ we {get} a closed local homology theory  $E\circ F_{\cC,\cU/2}\colon \UBC\to \bC$.
  The coarsification
  $(E\circ F_{\cC,\cU/2})\bP$ is equivalent to the coarse homology theory $QE$, which is the coarsification of $E$ from \cite[Def.~7.44]{buen}.
\hB
\end{ex}
   
\begin{rem}
Proposition \ref{foiwfewfewewfewwe} is also true (with a slightly different argument for excision using open tubular neighbourhoods of the subsets and in addition homotopy invariance) if we work with  $\Sp\cB^{o}$ instead of $\Sp\cB$.
\hB
\end{rem}

Using the open version we can define the coarsification 
$$E\bP^{o} \colon \BC\stackrel{\bP^{o}}{\to}  \Sp\cB^{o}\stackrel{E}{\to} \bC\, .$$
for open local homology theories $E$.

\begin{lem}\label{vweoijiorjovbwvdfsv}
If $E\colon \UBC\to \bC$ is a closed and open local homology theory, then we have an equivalence
$E\bP^{o}\simeq E\bP$.
\end{lem}

\begin{proof}
Let $X$ be in $\BC$ with coarse structure $\cC$. We can view  $E$ as a colimit preserving functor $\Sp\cB\to \bC$ and {also}
$\Sp\cB^{o}\to \bC$. The chain
\begin{eqnarray*}
E\bP^{o}(X)&\simeq&E(\colim_{U\in \cC} \Yo\cB^{s,o}(P_{U}(X)))\\&
 \simeq &\colim_{U\in \cC} E(P_{U}(X))\\&\simeq&
 E(\colim_{U\in \cC} \Yo\cB^{s}(P_{U}(X)))\\&\simeq&
 E\bP(X)
\end{eqnarray*}
gives now the desired equivalence.
\end{proof}

\section{From coarse to local homology theories via \texorpdfstring{$\bF$}{F}}

In this section we refine the forgetful functor
$F_{\cU} \colon \UBC\to \BC$ from \eqref{eqkj4zrtet} to a local homology theory. We define
\begin{equation}\label{wergergergergergweferferfw}
\bF \coloneqq \Yo^{s}\circ F_{\cU}\colon \UBC\to \Sp\cX \, .
\end{equation}

\begin{lem}\label{goigorgregregregre}
$\bF$ is a closed and open local homology theory.
\end{lem}

\begin{proof} 
The proof is straightforward and similar to the one of Lemma~\ref{fiowefewfewfwe}.
\end{proof}

We get a colimit-preserving functor 
\begin{equation}\label{ewfewfwffefwefe}
\bF\colon \Sp\cB\to \Sp\cX\,.
\end{equation}
 
 For a $\bC$-valued  coarse homology theory
 $E $ we write
 $$E\bF \colon\Sp\cB\xrightarrow{\bF} \Sp\cX \xrightarrow{E} \bC$$
 for the associated local homology theory (compare with Lemma~\ref{fiowefewfewfwe} where the notation $E\circ F_{\cU}$ was used).
 
\begin{prop}\label{frioorfwefewfefwefew}
We have a canonical equivalence
\begin{equation}
\label{eq345ztrfd}
\id \xrightarrow{\simeq} \bF\circ \bP\, .
\end{equation}
\end{prop}

\begin{proof}
We have a functor $$I \colon \BC^{\cC}\to \BC$$
that is defined on objects by $I(X,U)\coloneqq X_{U}$.
By $u$-continuity of $\Yo^{s}$ the left Kan extension of $\Yo^{s}\circ I$ along $F_{\cC}$ {(see \eqref{eqjk3re} for the definition of $F_{\cC}$)} is equivalent to $\Yo^{s}$.  
Let $(X,U)  $ be in $\BC^{\cC}$. Dirac measures provide a canonical inclusion
$X\to P_{U}(X)$ of sets. This map is an equivalence  
\begin{equation}\label{vrvheioveirvioerververvrev}
X_{U} \xrightarrow{\simeq} F_{\cU}(P(X,U))
\end{equation} 
 of bornological coarse spaces. 
 Hence we get an equivalence of functors from $\BC^{\cC}$ to $\Sp\cX$
 \begin{equation}\label{ggoiugo3ig3433g}
\Yo^{s}\circ I \xrightarrow{\simeq}   \Yo^{s}\circ F_{\cU}\circ P \simeq \bF\circ \Yo\cB^{s} \circ P \, .
\end{equation} 
Since $\bF$ is colimit preserving, the equivalence \eqref{ggoiugo3ig3433g} induces an equivalence of left Kan extensions along $F_{\cC}$:
$$\Yo^{s} \xrightarrow{\simeq} \bF\circ \bP\colon \BC\to \Sp\cX\, .$$
We finally interpret $\bP$ as a colimit preserving functor
$\Sp\cX\to \Sp\cB$ to get the desired equivalence.
\end{proof}

\begin{kor}
Every coarse homology theory $E$ is equivalent to the coarsification of the closed local homology theory $E\bF$. Similarly, every morphism between coarse homology theories is induced by coarsification from a morphism between the associated closed local homology theories.
\end{kor}

\section{Coarsifying spaces}\label{fewwpowepkfpewofewfeww}

Under certain finiteness conditions on the  uniform bornological coarse space $X$ we can construct a morphism
$$c_{X} \colon \Yo\cB^{s}(X)\to \bP(\bF(X))$$
called the comparison morphism. {We will furthermore show that it is an equivalence for simplicial complexes of bounded geometry which are uniformly contractible. Part of the material here is inspired by Roe \cite[Ch.~2, Part ``Coarse algebraic topology'']{roe_index_coarse}.}

Let $X$  be a coarse space with a uniform structure. 

\begin{ddd}\label{fewoifwefewfewf}
We say that  the uniform structure is   numerable if there is an 
entourage~$U$ which is both coarse and uniform,  and   an equicontinuous and  uniformly  point-wise  locally finite partition
of unity $(\chi_{\alpha})_{\alpha\in A}$ such that $\supp(\chi_{\alpha})$ is $U$-bounded for all $\alpha $ in $A$.
\end{ddd}

\begin{rem}
Here \emph{uniform  point-wise  local finiteness} means that 
$$\sup_{x\in X}|\{\alpha\in A\:|\: \chi_{\alpha}(x)\not=0\}| < \infty\, .$$

The condition  of  \emph{equicontinuity} requires  that for  every  positive real number $\varepsilon$   there exists a uniform entourage $V$ of $X$ such that for all $\alpha$ in $A$ and $(x,x^{\prime})$ in $V$ we have the inequality
$|\chi_{\alpha}(x)-\chi_{\alpha}(x^{\prime})|\le \varepsilon$.
\hB
\end{rem}

Let $X$ be a simplicial complex with the coarse and uniform structures both induced from the spherical path quasi metric. 
\begin{lem}\label{fijwjifofewfewfwf}
If $X$ is finite-dimensional, then $X$ is numerable.
\end{lem}

\begin{proof}
We consider the entourage $U_{2}$ of width $2$.
We define the equicontinuous partition of unity
$(\chi_{v})_{v\in X^{(0)}}$  using the baricentric coordinates of the simplices,   where $X^{(0)}$ is the set of vertices of $X$.
If $\sigma$ is a simplex in $X$ and $x$ is a point in the interior of $\sigma$, then
$\chi_{v}(x)\not=0$ exactly if $v$ is a vertex of $\sigma$. 
Hence for every point $x$  the number of vertices $v$ of $X$ with $\chi_{v}(x)\not=0$ is bounded by $\dim(X)
+1$.

The support of $\chi_{v}$ is $U_{2}$-bounded for every vertex $v$ of $X$.
\end{proof}


Let $X$ be a numerable uniform bornological coarse space. By numerability of the uniform structure  we can choose an entourage  $U$  which is coarse and uniform    such that
there exists an equicontinuous and uniformly  point-wise  locally finite   partition of unity $(\chi_{\alpha})_{\alpha\in A}$ on $X$ such that
$\supp(\chi_{\alpha})$ is $U$-bounded for every $\alpha$ in $A$.  We choose a family of points $(x_{\alpha})_{\alpha\in A}$ in $X$ such that $x_{\alpha}\in \supp(\chi_{\alpha})$ for all $\alpha$ in $A$.
We can then define a map (not necessarily controlled)
\begin{equation}
\label{eqjk34trsd32}
X\to P_{U^{2}}(F_{\cU}(X))\, , \quad x\mapsto \sum_{\alpha\in A} \chi_{\alpha}(x) \delta_{x_{\alpha}}\, .
\end{equation}
  This map is  uniform. Note that at this point
  we use the uniformity of   the  point-wise local finiteness condition, because we measure distances in the simplices of  $P_{U^{2}}(F_{\cU}(X))$ in the spherical metric and not in the maximum metric with respect to baricentric coordinates, cf.~\cite[Ex.~5.37]{buen}.
  
The map defined {in \eqref{eqjk34trsd32}} can also be regarded as a morphism of  uniform bornological coarse spaces
$\tilde c:X_{U}\to P_{U^{2}}(F_{\cU}(X))$.
It induces a compatible system of  morphisms
$$\Yo\cB^{s}(X_{U})\to  \Yo\cB^{s}(P_{U^{2}}(F_{\cU}(X)))\to \bP(\bF(X))$$ for all  sufficiently large entourages $U$ of $X$,  and by $u$-continuity of $\Yo \cB^{s}$, a morphism
$$c_{X}\colon \Yo\cB^{s}(X)\to  \bP(\bF(X))\, .$$

\begin{ddd}\label{defn234eds}
For a numerable uniform bornological coarse space the morphism $c_{X}$ is called the comparison map.
\hB
\end{ddd}

%
%

\begin{rem}\label{ijfoewewfewfwefwef}
We must assume that $X$ is numerable in order to produce a uniform map $X\to P_{U^{2}}(F_{\cU}(X))$ by \eqref{eqjk34trsd32}. 

In the classical approach to the coarsification of locally finite homology theories (see, e.g., Higson--Roe \cite[Sec.~3]{hr}) one only needs a coarse and continuous map. In this case the same formula works, and we only have to assume that the members of the partition of unity have uniformly controlled support. The existence of such a partition of unity follows from the compatibility of the uniform and the coarse structure if we in addition assume  that the underlying topological space of  $X$ is paracompact.

In our approach we must work with uniform maps since this is required by functoriality of the cone functor $\cO$ which we employ below in order to construct the assembly map.
\hB
\end{rem}

\begin{lem}
Up to equivalence the comparison map does not depend on the choice of the partition of unity.
\end{lem}

\begin{proof}
We consider a second choice of partition of unity (without loss of generality for the same entourage $U$) and denote the associated morphism by
$\tilde c^{\prime}\colon X_{U}\to P_{U^{2}}(F_{\cU}(X))$.
Then $\tilde c$ and $\tilde c'$ are $\pi/2$-close to each other and
$s\mapsto (1-s)  \tilde c +s\tilde c^{\prime}$
is a uniform homotopy between $\tilde c$ and $\tilde c^{\prime}$.
  We now use that $\Yo\cB^{s}$ is homotopy invariant.
\end{proof}

Let $f\colon X\to X^{\prime}$ be a morphism of uniform bornological coarse spaces which are assumed to be numerable.
\begin{lem}
We have an equivalence 
$$c_{X^{\prime}} \circ  \Yo\cB^{s}(f)\simeq (\bP\circ \bF)(f) \circ c_{X}\, .$$
\end{lem}

\begin{proof}
After choosing partitions of unity  for $X$ and $X^{\prime}$ with bounds $U$ and $U^{\prime}$ such that $(f\times f)(U)\subseteq U^{\prime}$
we have a square  (not necessarily commuting) of morphisms of uniform bornological coarse spaces
\[\xymatrix{
X_U \ar[rr]^-{\tilde c_{X}}\ar[d]^{f} & & P_{U^2}(F_{\cU} (X))\ar[d]^{P(F_{\cU}(f))}\\
X^{\prime}_{U^\prime} \ar[rr]^-{\tilde c_{X^{\prime}}} & & P_{U^{\prime,2}}(F_{\cU}(X^{\prime}))
}\]
We now observe that the compositions
$P(F_{\cU}(f))\circ \tilde c_{X}$ and $\tilde c_{X^{\prime}}\circ f$ are
close to each other and (linearly) homotopic. Hence they become equivalent after application of $\Yo\cB^{s}$.
\end{proof}

%

Let $Y$ be a uniform bornological coarse space.
 
\begin{ddd}\label{wefiowef34t6453}
We say that $Y$ is coarsifying if it is numerable and  the comparison map $c_{Y}$
 is an equivalence. 
\end{ddd}

Let $E$ be a local homology theory.
If $Y$ is coarsifying, then the comparison map induces   an equivalence
\[E(c_{X})\colon   E(Y)\xrightarrow{\simeq} E\bP(F_{\cU}(Y))\, .\]

Let $X$ be a numerable uniform bornological coarse space.
\begin{ddd}
A morphism $f\colon X\to Y$ in $\UBC$ is called a coarsifying approximation if
$Y$ is coarsifying and {$(\bP\circ \bF)(f)$}
is an equivalence.
\end{ddd}

Let $E$ be a local homology theory. If $X\to Y$ is a coarsifying approximation,
then by construction we have an equivalence  
$$E\bP(F_{\cU}(X))\simeq E(Y)\, .$$

In the following we discuss  an important class of examples of coarsifying spaces, see also \cite[Sec.~7.4]{buen}.

Below  $B^{q+1}$ is the unit ball in $\IR^{q+1}$ and $S^q$ is its boundary.
\begin{ddd}
\begin{enumerate}
\item A simplicial complex $K$ has bounded geometry if the number of vertices in the stars of its vertices is uniformly bounded.
\item A metric space $X$ is equicontinuously contractible, if for every $q$ in $\IN$ and for every equicontinuous family of maps $\{\varphi_i\colon  S^q \to X\}_{i \in I}$ there exists an equicontinuous family of maps $\{\Phi_i \colon  B^{q+1} \to X\}_{i \in I}$ with ${\Phi_i}|_{\partial B^{q+1}} = \varphi_i$.\footnote{This is a slight strengthening of the notion of uniform contractibility which is commonly used in the coarse geometry literature.}
\end{enumerate}
\end{ddd}

 Let $K$ be a simplicial complex. We get a \ubs $K_{d}$ by equipping $K$ with the  bornology of bounded subsets and the metric coarse and uniform structures.

Let $A$  be a subcomplex of $K$, $X$ be a metric space and $f \colon  K_d \to X_d$ be a morphism of bornological coarse spaces such that $f|_A$ is uniformly continuous.

\begin{lem}\label{lem1243treweq} If $K$ is finite-dimensional and $X$ is equicontinuously contractible, then $f$ is close to a morphism of uniform bornological coarse spaces which extends $f|_A$, and which is in addition uniformly continuous.
 \end{lem}

\begin{proof}
The proof given in \cite[Lem.~7.72]{buen} (which covers the non-uniform version of this lemma) also works literally here.
\end{proof}

Let $K$ and $K'$ be two simplicial complexes and $f\colon K_{d}\to K'_{d}$ be a morphism between the underlying bornological coarse spaces.
\begin{lem}\label{oigjwtgwergergwgg}
If $K$ and $K'$ are equicontinuously contractible and $f$ is a coarse equivalence, 
then $f$ is close to a homotopy equivalence in $\UBC$ and any two such homotopy equivalences are homotopic to each other.
\end{lem}
\begin{proof}
The proof  is the same as for \cite[Lem.~7.73]{buen}
where we use Lemma~\ref{lem1243treweq} instead of Lemma   \cite[Lem.~7.72]{buen}
in order to get the additional uniformity.
\end{proof}

\begin{prop}\label{prop34redsfg}
If $K$ is a simplicial complex of bounded geometry which is equicontinuously contractible as a metric space, then the uniform bornological coarse space $K_{d}$ is coarsifying.
\end{prop}

\begin{proof}
Note that $K$ is finite-dimensional and hence $K_{d}$ is numerable by Lemma~\ref{fijwjifofewfewfwf}. The verification that the comparison map for $K_{d}$ is an equivalence is the core of the argument of \cite[Prop.~7.80]{buen}, which is itself taken from Nowak--Yu \cite[Proof of Thm.~7.6.2]{nowak_yu}. 

As in the beginning of the proof of \cite[Prop.~7.80]{buen} 
we construct a diagram of maps
\[\xymatrix{
K_{d} \ar@/^1pc/[r]^-{f} & P_U(F_{\cU}(K^{(0)}_{d}))\ar@/^1pc/[r]^-{f_0} \ar[l]^<{g} & P_{U_1}(F_{\cU}(K^{(0)}_{d})) \ar@/^1pc/[r]^-{f_1} \ar@/^1.3pc/[ll]^<{g_0} & P_{U_2}(F_{\cU}(K^{(0)}_{d})) \ar@/^1pc/[r]^-{f_2} \ar@/^2.5pc/[lll]^<{g_1} & \cdots
}\]
with $g_{n}\circ (f_{n}\circ \dots\circ f_{0}\circ f)$ being homotopic in $\UBC$ to $\id_{K_{d}}$, and $(f_{n}\circ \dots\circ f_{0}\circ f)\circ g_{n-1}$ being homotopic to $f_{n}$ in $\UBC$. 
Here we use the  Lemmas~\ref{lem1243treweq} and \ref{oigjwtgwergergwgg} instead of \cite[Lem.~7.72 and 7.73]{buen}. The stronger assumption  that $K$ {is} equicontinuously contractible (instead of {just} uniformly contractible {as} in \cite{buen}) implies that the resulting maps are uniformly continuous (instead of {just} continuous {as} in \cite{buen}).

We claim that the induced comparison map
$$\Yo^{s}\cB(K_{d})\to \colim_{n\in \nat} \Yo^{s}\cB(P_{U_{n}}(F_{\cU}(K^{(0)}_{d}))) \simeq \bP(\bF(K^{(0)}_{d}))$$
is an equivalence. 
This is in fact an instance of the following general fact\footnote{We thank the referee for suggesting this simple argument.}. Assume that we have a diagram
$$A_{0}\xrightarrow{f_{0}} A_{1} \xrightarrow{f_{1}} A_{2} \xrightarrow{\phantom{f_1}} \dots$$
in a stable $\infty$-category such that the maps $A_{0}\to A_{i}$ admit retracts $g_{i}$ for all $i$ in $\nat$ and such that $f_{i}$ is equivalent to $A_{i}\xrightarrow{g_{i}} A_{0}\to A_{i+1}$.
Since in a stable $\infty$-category retracts split as sums the diagram
is equivalent to a sum of a constant diagram build from $A_{0}$ and identity maps and a diagram with zero transition maps. This implies that
$A_{0}\to \colim_{\nat} A_{i}$ is {an} equivalence.

The inclusion
$K_{d}^{(0)}\to K_{d}$ is an equivalence of the underlying bornological coarse spaces and therefore induces the second equivalence in
\[
\Yo^{s}\cB(K_{d})\xrightarrow{\simeq} \bP(\bF(K^{(0)}_{d}))\xrightarrow{\simeq} \bP(\bF(K_{d}))\,.\qedhere
\]
\end{proof}

\begin{ex}
The following is taken from \cite[Ex.~7.71]{buen} and originally goes back to Gromov \cite[Ex.~1.D$_1$]{gromov}:

Let $G$ be a finitely generated group admitting a model for its classifying space $BG$ which is a finite simplicial complex. Then the universal cover $EG$ of $BG$ is a simplicial complex of bounded geometry which is {equicontinuously} contractible, i.e., $EG_d$ is coarsifying by the above Proposition~\ref{prop34redsfg}.

The group $G$ quipped with a word-metric becomes a metric space and hence a uniform bornological coarse space $G_d$. The action of $G$ on $EG$ provides a morphism $f\colon G_d \to EG_d$ in $\UBC$ which depends on the choice of a base-point in $EG$. The morphism $(\bP\circ \bF)(f)$  is an equivalence since $f$ is a coarse equivalence between the underlying bornological coarse spaces and hence $\bF(f)$ is an equivalence.  
Therefore we have shown that $f\colon G_d \to EG_d$ is a coarsifying approximation.
\hB
\end{ex}

\section{Cone functors}\label{blijeobereggreerg}

In this section we describe the cone functor $\cO\colon \UBC\to \BC$ and its germs at infinity $\cO^{\infty}\colon \UBC\to \Sp\cX$.
These functors play a crucial role in the construction of the coarse assembly map. After the introduction of the cone functor, we compare it with variants which occur in the literature on coarse geometry and which are useful in certain arguments.
 
In short,  the cone of a \ubs $X$ is the bornological coarse space $\cO(X)$  obtained from the bornological coarse space 
$F_{\cU}([0,\infty)\otimes  X) $ by replacing the coarse structure by the hybrid structure {(cf.~\cite[Sec.~5.1]{buen})}
associated to the family of subsets $\cY:=([0,n]\times X)_{n\in \nat}$ and the uniform structure on $[0,\infty)\otimes  X$.

In the following we spell out the definition of the cone explicitly.  
 Let $\cT$ denote the uniform structure of $X$. We consider $\cP_{X\times X}$  and its subset $\cT$ with the opposite of the inclusion relation.  By \cite[Def.~5.9]{buen} a function (i.e., an order preserving map) $\phi\colon [0,\infty)\to \cP_{X\times X}$ is called  $\cT$-admissible if for every uniform entourage $U$ in $\cT$ there exists an element $t$ in $[0,\infty)$ such that $\phi(s)\subseteq U$ for all $s$ in $[t,\infty)$.

\begin{ddd}\label{rgfporgergergereg1}
We let $ \cO(X)$  be the bornological coarse space defined as follows:
\begin{enumerate}
\item The underlying set of $  \cO(X)$ is $[0,\infty)\times X$.
\item The bornology of $  \cO(X)$ is generated by the subsets $[0,n]\times B$ for all $n$ in $\nat$ and bounded subsets $B$ of $X$. 
\item \label{fiewjioffwefwefwef2} The coarse structure of $ \cO(X)$
is generated by  the entourages of the form
$V\cap U_{(\kappa,\phi)}$,  where    $V$  is  a coarse entourage of $[0,\infty)\otimes X$ and
\[\mathclap{
U_{(\kappa,\phi)}:=\{((s,x),(t,y))\in ([0,\infty)\times X)^{2} \:|\: |s-t|\le \kappa(\max\{s,t\})\:\&\:  (x,y)\in  \phi(\max\{s,t\})\}\, .
}\]
for all  $\cT$-admissible functions
$\phi\colon [0,\infty)\to \cP_{X\times X}$ and  functions $\kappa\colon [0,\infty)\to [0,\infty)$ satisfying $\lim_{t\to \infty} \kappa(t)=0$.
\end{enumerate}
\end{ddd}

If $f:X\to X^{\prime}$ is a morphism of uniform bornological coarse spaces, then
the map
$$\id_{[0,\infty)}\times f:[0,\infty)\times X\to [0,\infty)\times X$$ is a morphism of bornological coarse spaces
$$\cO(f):\cO(X)\to \cO(X^{\prime})\, .$$
We thus have described the cone functor
$$\cO\colon \UBC\to \BC\, .$$
The maps of sets 
$X \to \{0\}\times X\to [0,\infty)\times X$ for $X$ in $\UBC$
induce a natural transformation
of functors
$$F_{\cU}\to \cO\colon \UBC\to \BC\, .$$
We apply $\Yo^{s}$ and take the cofibre in order to get a cofibre sequence (called the cone sequence)
\begin{equation}\label{23ru90rr}
 \Yo^{s}(F_{\cU}(X))\to \Yo^{s}(\cO(X))\to \cO^{\infty}(X)\to \Sigma \Yo^{s}(F_{\cU}(X))\, ,
\end{equation}
in $\Sp\cX$ which is functorial for $ X$  in $\UBC$,
 where, by definiton
 $$\cO^{\infty}(X)\coloneqq \Cofib\big(  \Yo^{s}(F_{\cU}(X)) \to \Yo^{s}(\cO(X))\big)\, .$$

\begin{ddd}\label{eroiergegggg}
 We call the resulting functor $\cO^{\infty}\colon \UBC\to \Sp\cX$
 the germs at infinity  of the cone.
 \end{ddd}

In order to connect with \cite[Sec.~5.2.3]{buen} note the following.
Let $X$ be a uniform bornological coarse space. Then $\cY(X):=([0,n]\times X)_{n\in \nat}$ is a big family in $\cO(X)$.
For every $n$ in $\nat$ the inclusion
$X\to \{0\}\times X\to [0,n]\times X$ induces a coarse equivalence,  and hence   an equivalence
$\Yo^{s}(F_{\cU}(X))\to \Yo^{s}(([0,n]\times X)_{\cO(X)})$  in $\Sp\cX$, where the subscript $\cO(X)$ indicates that the structures on the subset are induced from $\cO(X)$.
The collection of these equivalences for all $n$ in $\nat$ induces an equivalence 
$$\Yo^{s}(F_{\cU}(X))\simeq \Yo^{s}(\cY(X))$$ in $\Sp\cX$.
The pair sequence of $(\cO(X),\cY(X))$ is therefore equivalent to the cone sequence \eqref{23ru90rr},
in particular we have an equivalence 
$$\cO^{\infty}(X)\simeq (\cO(X),\cY(X))\, ,$$ where the right-hand side is interpreted as in \cite[(4.5)]{buen}.
   
We refer to \cite[Ex.~5.16]{buen} and \cite[Sec.~9]{equicoarse} for more details. In particular, in \cite[Prop. 9.31]{equicoarse}
we show that $\cO^{\infty}$ is represented by the bornological coarse space 
\begin{equation}\label{ewfoijfoiqjfoiqefe}
\cO^{\infty}_{\geom}(X):=((-\infty,0]\otimes X)\sqcup_{\{0\}\times X}\cO(X)\, , 
\end{equation}
where the push-out is interpreted in $\BC$.


 In the proof of Proposition \ref{ergioerogergregreg} below it is useful to use a modified version of the cone over a uniform bornological coarse space $X$ which we will denote by $\tilde \cO(X)$.   
\begin{ddd}\label{rgfporgergergereg}We let $\tilde \cO(X)$  be the bornological coarse space defined as follows:
\begin{enumerate}
\item The underlying set of $\tilde \cO(X)$ is $[0,\infty)\times X$.
\item The bornology of $\tilde \cO(X)$ is generated by the subsets $[0,n]\times B$ for all $n$ in $\nat$ and bounded subsets $B$ of $X$ 
\item \label{fiewjioffwefwefwef} The coarse structure of $\tilde \cO(X)$
is generated by the  entourages of the form
$V\cap U_{\phi}$,  where    $V$  is  a coarse entourage of $[0,\infty)\otimes X$ and
$$U_{\phi}:=\{((s,x),(t,y))\in ([0,\infty)\times X)\times ([0,\infty)\times X)\:|\: (x,y)\in  \phi(\max\{s,t\})\}$$
for all $\cT$-admissible   functions
$\phi\colon[0,\infty)\to\cP_{X\times X}$.  
\end{enumerate}
\end{ddd}
Note that the underlying bornological spaces of $\cO(X)$, $\tilde \cO(X)$ and   $[0,\infty)\otimes X$ coincide.
The identity map of the underlying sets induces a morphism \begin{equation}\label{rfviuhwiofwfwefwefewfef}
i \colon \cO(X) \to \tilde \cO(X)
\end{equation}
which is natural in $X$.

\begin{lem}\label{ewfiowefwefergergergerg}
The morphism \eqref{rfviuhwiofwfwefwefewfef} induces an equivalence
\[
\Yo^{s}(i) \colon \Yo^{s}(\cO(X))\to \Yo^{s}(\tilde \cO(X))\, .
\]
\end{lem}
\begin{proof}
We define a map of sets
$$q \colon [0,\infty)\times X\to[0,\infty)\times X\, , \quad q(t,x) \coloneqq (\sqrt{1+t},x)\, .$$
The map $q$ induces a morphism of bornological coarse spaces
$j \colon \tilde \cO(X)\to \cO(X)$.
Note that the compositions $i\circ j$ and {$j \circ i$} are both given {on the level of sets} by the map $q$.
It suffices to show that the morphisms on $\Yo^{s}(\cO(X))$ or $\Yo^{s}(\tilde \cO(X))$, respectively, induced by 
$q$  are equivalent to the respective identities.

We first consider the case of the modified cone $\tilde \cO(X)$. In this case we shall see that $q$ is coarsely homotopic to the identity (see \cite[Defn.~4.17]{buen}). In order to define the homotopy 
we let the map  $p_{+} \colon \tilde \cO(X)\to [0,\infty)$ be given by $$p_{+}(t,x) \coloneqq t+1-\sqrt{t+1}$$ and set $p \coloneqq (p_{+},0)$. 
Note that $p_{+}$ is bornological and controlled. Then we define the coarse homotopy
\begin{equation}\label{gbetookrepogergrege}
I_{p} \tilde  \cO(X)\to   \tilde\cO(X)\, , \quad (u,t,x)\mapsto \Big(\Big(1-\frac{u}{p_{+}(t)}\Big) t + \frac{u}{p_{+}(t)} \sqrt{1+t},x\Big)
\end{equation}
(see \cite[Defn.~4.14]{buen} for notation of coarse cylinders).
One easily checks that this map is  proper and controlled. 
Since $\Yo^{s}$ is invariant under coarse homotopies (in particular using \cite[Cor.~4.18]{buen}) we conclude that
$$\Yo^{s}(q) \colon \Yo^{s}(\tilde \cO(X))\to \Yo^{s}(\tilde \cO(X))$$ is equivalent to the identity.

The case of the cone $\cO(X)$ is more involved. By Definition \ref{rgfporgergergereg1}
 the hybrid structure on $\cO(X)$ is generated by entourages of the form
$V\cap U_{(\kappa,\phi)}$. 
   We fix the pair $(\kappa,\phi)$ and $V$. We can now choose a  differentiable function $\sigma \colon [0,\infty)\to [0,\infty)$ such that
 $\lim_{t\to \infty} \sigma(t)=\infty$ and $  p_{+} \colon \cO(X)_{V\cap U_{(\kappa,\phi)}}\to [0,\infty)$ given by
 $$p_{+}(t) \coloneqq \sigma(t)( t+1-\sqrt{t+1})$$ is controlled. To this end we must make sure that $(1+t)\sigma^{\prime}(t) $ and
$\sigma \kappa$ are both uniformly bounded.  Note that $  p_{+}$ is also bornological.
We then define the coarse homotopy  
$$I_{p}\cO(X)_{V\cap U_{(\kappa,\phi)}}\to \cO(X)$$ 
between the maps induced by $\id_{[0,\infty)\times X}$ and $q$
by the same formula as in \eqref{gbetookrepogergrege} as above. Indeed one checks that this map is proper and controlled.
Hence we have an equivalence  of morphisms
$$\Yo^{s}(q)\simeq \Yo^{s}(\id) \colon \Yo^{s}(\cO(X)_{V\cap U_{(\kappa,\phi)}})\to \Yo^{s}(\cO(X))\, .$$
We now perform the colimit of these equivalences over the poset of data $(V,(\kappa,\phi))$.
By $u$-continuity we get the desired equivalence of
$$\Yo^{s}(q) \colon \Yo^{s}(\cO(X))\to \Yo^{s}(\cO(X))$$ 
with the identity.
   \end{proof}

   Note that in the definition of the modified cone $\tilde \cO (X)$ we have not fixed the decay rate (encoded in the function $\phi$ in Definition \ref{rgfporgergergereg}.\ref{fiewjioffwefwefwef}) of the entourages in the $X$-direction as $t$ {and $s$} tend to $\infty$. Let us fix a $\cT$-admissible function
$\phi\colon [0,\infty)\to \cT $ which we assume to be monotone and such that
$\phi(0)=X\times X$. Note that here we assume that $\phi$ takes values in $\cT$ (instead of $\cP_{X\times X}$), and therefore $\cT$-admissibility is the same as cofinality.
\begin{ddd}\label{rgfporgergergereg222}
We let $\tilde \cO_{\phi}(X)$  be the bornological coarse space defined as follows:
\begin{enumerate}
\item The underlying set of $\tilde \cO_{\phi}(X)$ is $[0,\infty)\times X$.
\item The bornology of $\tilde \cO_{\phi}(X)$ is generated by the subsets $[0,n]\times B$ for all  $n$ in $\nat$ and bounded subsets $B$ of $X$.
\item \label{fiewjioffwefwefwef11} The coarse structure of $\tilde \cO_{\phi}(X)$
is generated by  entourages of the form
$V\cap U_{\phi}$,  where    $V$  is  a coarse entourage of $[0,\infty)\otimes X$.
\end{enumerate}
\end{ddd}

\begin{ex}
{Let $X$ be a metric space. Recall that its coarse structure is generated by the collection of entourages $W_{r}:=\{(x,y)\in X\times X\:|\: d(x,y)\le r\}$ for all $r > 0$. If we set $\phi(t)=W_{1/t}$, then $\tilde \cO_{\phi}(X)$ is the open cone over $X$ as considered at many places in the coarse geometry literature and usually called the Euclidean cone over $X$.}
\hB
\end{ex}

We have a canonical morphism
\begin{equation}\label{gijgo343f34f3}
k_{\phi} \colon \tilde \cO_{\phi}(X)\to \tilde \cO(X)
\end{equation}
 given by the identity of the underlying sets.
\begin{lem}\label{greij3o4it3434434343t}If $\phi$ is monotone and satisfies $\phi(0)=X\times X$, then
the map \eqref{gijgo343f34f3} induces an equivalence
$$\Yo^{s}(k_{\phi}) \colon \Yo^{s}(\tilde \cO_{\phi}(X))\to\Yo^{s}(\tilde \cO(X))\, .$$
\end{lem}
\begin{proof}
If $\phi^{\prime}$ is a second monotone  function as in \ref{rgfporgergergereg}.\ref{fiewjioffwefwefwef} such that  $\phi(t)\subseteq \phi^{\prime}(t)$ for all $t$ in $[0,\infty)$, then $U_{\phi}\subseteq U_{\phi^{\prime}}$. Therefore the identity of the underlying maps induces a morphism
$$k_{\phi}^{\phi^{\prime}} \colon \tilde \cO_{\phi}(X)\to \tilde \cO_{\phi^{\prime}}(X)\, .$$
By $u$-continuity we have an equivalence 
$$\Yo^{s}(\tilde\cO(X))\simeq \colim_{\phi^{\prime}\ge \phi} \Yo^{s}(\tilde \cO_{\phi^{\prime}}(X))\, .$$
It therefore suffices to show that
$$\Yo^{s}(k_{\phi}^{\phi^{\prime}}) \colon \Yo^{s}(\tilde \cO_{\phi}(X))\to \Yo^{s}(\tilde \cO_{\phi^{\prime}}(X))$$
is an equivalence for all pairs $\phi,\phi^{\prime}$  such that  $\phi(t)\subseteq \phi^{\prime}(t)$ for all $t$ in $[0,\infty)$.

{We will show now that there exists} a controlled function  $\sigma \colon [0,\infty)\to [0,\infty)$ such that $\phi^{\prime}(t)\subseteq \phi(\sigma(t))$ and $\lim_{t\to \infty} \sigma(t)=\infty$. 
To this end set $$\delta:[0,\infty)\to [0,\infty)\, , \quad \delta(s) \coloneqq \sup\{t\in [0,\infty) \:|\:   \phi^{\prime}(s)\subseteq \phi(t)\}\, .$$
This function is monotonically increasing and satisfies $\lim_{s\to \infty } \delta(s)=\infty$.
The idea is now  to define $\sigma$ to be $\delta$. But to ensure that $\sigma$ is controlled, we have to modify this idea slightly. We choose $t_{0}\in [0,\infty)$ such that $\delta(t_{0})\ge 2$. We can find $\sigma(t)$ for all $t$ in $ [t_{0},\infty)$ by solving the  equation
$$t=\int_{0}^{\sigma(t)} h(s) ds\, ,$$
where $h$ is a function with $h\ge 1$ and
$$t\le \int_{0}^{\delta(t)} h({s})ds\, .$$
{More concretely,} we can take
$$h(t) \coloneqq \max\big\{1, \sup_{s\in [1,t]} \delta^{-1}(2s)\big\}  \, ,$$
where we set $\delta^{-1}(u) \coloneqq \sup\{r\in [0,\infty) \:|\: \delta(r)\le u\}$.
Note  that if $ t\in [t_{0},\infty)$, then    the interval $[\delta(t)/2,\delta(t)]$ in the domain of integration yields the estimate
$$ \int_{0}^{\delta(t)} h({s})ds\ge \frac{\delta(t)}{2}  \delta^{-1}(2\delta(t)/2) \ge t\, .$$
 For $t\in [0,t_{0}]$ we set   $\sigma(t)=0$.
The Lipschitz constant of $\sigma$ on $[t_{0},\infty)$ is bounded by $1$.
It follows that $\sigma$ is controlled.

We consider the map of sets $$q \colon [0,\infty)\times X\to [0,\infty)\times X\, , \quad  q(t,x) \coloneqq (\sigma(t),x)\, .$$
By construction it induces a morphism
$$j \colon \tilde \cO_{\phi^{\prime}}(X)\to \tilde \cO_{\phi}(X)\, .$$
We now note that the compositions
$$j\circ k_{\phi}^{\phi^{\prime}} \colon  \tilde \cO_{\phi}(X)\to  \tilde \cO_{\phi}(X)\, , \quad k_{\phi}^{\phi^{\prime}}\circ j \colon 
\tilde \cO_{\phi^{\prime}}(X)\to \tilde \cO_{\phi^{\prime}}(X)$$ are both induced by 
$q$. 

It suffices to show that these morphisms are both coarsely homotopic to the identity.

We set $p:=(\sigma+1,0)$ and observe that the map
$$I_{p}\tilde \cO_{\phi}(X) \to \tilde \cO_{\phi}(X)\, , \quad (u,t,x)\mapsto \Big(\Big(1-\frac{u}{\sigma(t)+1}\Big)t+\frac{u}{\sigma(t)+1}\sigma(t),x\Big)$$
{is a suitable homotopy (i.e., proper and controlled) that does the job.}
The same construction also works in the case of $\phi^{\prime}$.
\end{proof}

\begin{rem}
The cone $\tilde \cO(X)$ has a big family $\cY(X):=([0,n]\otimes X)_{n\in \nat}$ and we can define a modified version of the germs at infinity 
$$\tilde \cO^{\infty}(X):=\Cofib\big( \Yo^{s}( \cY(X))\to  \Yo^{s}(\tilde \cO(X))\big) \, .$$
Similarly we can define $$\tilde \cO_{\phi}^{\infty}(X):=\Cofib\big( \Yo^{s}( \cY(X))\to  \Yo^{s}(\tilde \cO_{\phi}(X))\big) \, .$$
 The inclusion $X\to [0,n]\times X $  is a coarse equivalence for every $n$ in $\nat$ and the structure induced by
  $\tilde \cO(X)$ or $\tilde \cO_{\phi}(X)$, {respectively.} In the latter case this is granted by the condition that $\phi(0)=X\times X$.
  Therefore we get fibre sequences 
  $$\Yo^{s}(F_{\cU}(X))\to \Yo^{s}(\tilde \cO(X))\to\tilde \cO^{\infty}(X)\to \Sigma \Yo^{s}(F_{\cU}(X))\,,$$ and 
  $$\Yo^{s}(F_{\cU}(X))\to \Yo^{s}(\tilde \cO_{\phi}(X))\to\tilde \cO_{\phi}^{\infty}(X)\to \Sigma \Yo^{s}(F_{\cU}(X))\, ,$$
   respectively.
By a comparison with the cone sequence \eqref{23ru90rr} and by Lemmas \ref{ewfiowefwefergergergerg} and \ref{greij3o4it3434434343t} we get induced equivalences
$$\cO^{\infty}(X)\simeq \tilde \cO^{\infty}(X)\simeq \tilde \cO_{\phi}^{\infty}(X)\, .$$
So we could have defined the germs at infinity of the cone using a modified version of the cone. But since the modified cones do not come from a hybrid structure construction we can not apply the general theorems (Homotopy Theorem and Decomposition Theorem) for hybrid spaces shown in \cite[Sec.~5.2 \& 5.3]{buen} in order to deduce the properties of this functor, see e.g.\ Lemma~\ref{foprgfregegr} below. For this reason we prefer to work with $\cO(X)$ instead of $\tilde \cO(X)$ or $\tilde \cO_\phi(X)$.
\hB
\end{rem}

\begin{ex}
Let $X$ be a geodesic, locally compact hyperbolic metric space. One can construct a nice compactification of $X$ by attaching the Gromov boundary $\partial X$. Note that $\partial X$ is a compact metric space. Higson--Roe \cite{hr} showed that $X$ is coarsely homotopy equivalent to the Euclidean cone $\tilde \cO_{\phi}(\partial X)$ over its Gromov boundary $\partial X$. Together with the results of the present section we therefore get the equivalence
 \begin{equation}\label{foijoi3f3f}
\Yo^{s}(X) \simeq \Yo^{s}(\cO(\partial X))\, .
\end{equation}

Fukaya--Oguni \cite{fukaya_oguni} generalized the result of Higson--Roe to all proper coarsely convex spaces (examples are hyperbolic spaces, $\mathrm{CAT}(0)$ spaces and systolic complexes). Especially, we have the equivalence \eqref{foijoi3f3f} where $\partial X$ is a suitable version of Gromov's boundary.
\hB
\end{ex}

\section{The coarse assembly map}

In this section we define the coarse assembly map. 

Taking the functoriality of {the} cone sequence  \eqref{23ru90rr} into account and using the notation  \eqref{wergergergergergweferferfw}  we get 
 a fibre sequence of functors from $\UBC$ to $\Sp\cX$
\begin{equation}\label{gerpogikrepoergerg}
\bF \to \Yo^{s}\circ \cO\to \cO^{\infty} \xrightarrow{\partial} \Sigma \bF
\end{equation} 
which we call the cone sequence.  

\begin{rem}
The  cone boundary map $\partial$ in the cone sequence \eqref{gerpogikrepoergerg} has a very nice interpretation as a forget-control map.  For $X$ in $\UBC$ the identity of the underlying sets induces a map
\[
\partial_{\geom}\colon \cO^{\infty}_{\geom}(X)\to \R\otimes \cF_{\cU}(X)\, ,
\]
see \eqref{ewfoijfoiqjfoiqefe} for the domain.
The difference between the domain and the target is that the domain has a smaller coarse structure.
By   \cite[Prop.~9.31]{equicoarse} the induced map 
\[
\Yo^{s}(\partial_{\geom})\colon \Yo^{s}(\cO^{\infty}_{\geom}(X))\to \Yo^{s}(\R\otimes \cF_{\cU}(X))\simeq \Sigma \bF(X)
\]
is equivalent to the cone boundary.
\hB
\end{rem}

\begin{lem}\label{foprgfregegr}
The functors $\Yo^{s}\circ\cO, \cO^{\infty} \colon \UBC\to \Sp\cX$ satisfy excision for uniformly and coarsely excisive decompositions, and they are homotopy invariant.
\end{lem}

\begin{proof}
This is shown in \cite[Sec.~9.4 \& 9.5]{equicoarse}. 
\end{proof}

\begin{rem}
Since we consider excision for decompositions which are uniform and coarse at the same time it is not necessary to assume that our uniform spaces are Hausdorff, see \cite[Rem.~9.28]{equicoarse}.

Note further that we do not have to add adjectives like \textit{open} or \textit{closed} in the assumptions of Lemma \ref{foprgfregegr}.
\hB
\end{rem}

By Lemma~\ref{goigorgregregregre} {the} functor $\bF$ vanishes on flasque spaces, but we do not expect that $\cO$ vanishes on {them}. Assume that $X$ is a flasque \ubs with flasqueness witnessed by the self-map $f$. Then in general $\cO(f)$ is not close to the identity, nevertheless $\Yo^{s}(\cO(f))$ is equivalent to $\Yo^{s}(\cO(\id_{X}))$ \cite[Cor.~5.31]{buen}. In fact, the map $\cO(f)$ exhibits the cone $\cO(X)$ as a weakly flasque bornological coarse space in the sense of \cite[Def.~4.18]{equicoarse}, see \cite[Proof of Prop.~11.22]{equicoarse}.

\begin{ddd}[{\cite[Def.~4.19]{equicoarse}}]
\label{defnk34rte}
A coarse homology theory is called strong if it vanishes on weakly flasque bornological coarse spaces.
\end{ddd}

\begin{ex}\label{wrgfiowrgegergegergerg111}
Here is a list of examples of coarse homology theories which are strong: 
\begin{enumerate}
\item Coarse ordinary homology  \cite{equicoarse}.
\item Coarse algebraic $K$-homology with coefficients in an additive category   \cite{equicoarse}.
\item Coarse Waldhausen $K$-homology of spaces    with coefficients in a space \cite{Bunke:aa}.
\item Coarse algebraic $K$-homology with coefficients in a left-exact $\infty$-category  \cite{unik}.
\item Coarse topological $K$-homology with coefficients in a $C^{*}$-category   {\cite{bu}}.
\end{enumerate}
 In these references we actually consider the equivariant case for a group $G$. For the present application we just need the case of the trivial group $G=\{1\}$. 
\hB
\end{ex}


Let $\bC$ be a cocomplete stable $\infty$-category, and let $E \colon \BC\to \bC$ be a  coarse homology theory.  We consider the compositions $$E\cO^{\infty} \colon  \UBC\xrightarrow{\cO^{\infty}} \Sp\cX\xrightarrow{E} \bC\, , \quad E\cO \colon    \UBC\xrightarrow{\cO} \BC\xrightarrow{E} \bC\, ,$$
where for $E\cO^{\infty}$ we interpret $E$ as a colimit preserving functor
$E \colon \Sp\cX\to \bC$, see \cite[Cor.~4.24]{buen}.
\begin{lem}\label{fewoiiofwefewf}
If $E$ is strong, then the functors
$$E\cO^{\infty} ,   E\cO  \colon \UBC\to \bC$$
are closed and open local homology theories.
\end{lem}

\begin{proof}
For a \ubs $X$ we have a natural fibre sequence
\begin{equation}
\label{eq2ewrr}
E(F_{\cU}(X))\to E\cO(X)\to E\cO^{\infty}(X) {\xrightarrow{E(\partial)}} \Sigma E(F_{\cU}(X))\, .
\end{equation}
By Lemma \ref{foprgfregegr} 
both functors $E  \cO$ and $E\cO^{\infty}$ are homotopy invariant and satisfy excision.
It remains to show that $   E  \cO$ and $E\cO^{\infty}$ are $u$-continuous and vanish on flasques.  

By Lemma~\ref{fiowefewfewfwe} the functor $E \circ F_{\cU}$ is a closed  and open local homology theory. In particular it is $u$-continuous and vanishes on flasques.

The functor $\cO^{\infty}$ is invariant under  coarsenings (\cite[Prop.~9.33]{equicoarse} and Definition~\ref{gfiorjgoergre345}) which implies the equivalence
 $\cO^{\infty}(X_{U})\stackrel{\simeq}{\to} \cO^{\infty}(X)$
 for sufficiently large entourages $U$ of $X$. In particular, the functor $E\cO^{\infty}$ is $u$-continuous. It follows from the fibre sequence \eqref{eq2ewrr} that $E\cO$ is $u$-continuous.

If $X$ is flasque, then   $\cO(X)$ is weakly flasque.
Since $E(F_{\cU}(X))\simeq 0$ and also $E\cO(X)\simeq 0$ due to strongness of $E$, we conclude, using the fibre sequence    \eqref{eq2ewrr},  that $E\cO^{\infty}(X)\simeq 0$.
\end{proof}

Let $E$ be a strong coarse homology theory.
We first post-compose the  cone sequence \eqref{gerpogikrepoergerg} with $E$.
In view of the Lemmas  \ref{fewoiiofwefewf} and \ref{fiowefewfewfwe} we get a fibre sequence of closed and open local homology theories 
\begin{equation*}
E\bF\to E\cO\to E\cO^{\infty} \xrightarrow{E(\partial)} \Sigma E\bF \, .
\end{equation*}
We now pre-compose this fibre sequence with the  the universal coarsification functor $\bP$ from Definition \ref{ergioerggrg546546}
and get a fibre sequence of coarse homology theories 
 \begin{equation}\label{fwefqwefwefqwfqwefqwefqwefqewff}
E\bF\bP\to E\cO\bP\to E\cO^{\infty}\bP\xrightarrow{E(\partial) \bP} \Sigma E\bF\bP\, .
\end{equation}

Let $E$ be a strong coarse homology theory.

\begin{ddd}\label{fiwjfofewfewfwefefweffw}
The coarse assembly map is the  natural transformation between  coarse homology theories
$$\mu_{E}  \colon  E\cO^{\infty}\bP\to \Sigma E$$
defined as the composition of $E(\partial)  \bP$ with  the identification $E\bF\bP \simeq E$ from \eqref{eq345ztrfd}.
\end{ddd}

\begin{rem}
Note that the universal coarsification functor $\bP$ takes values in $\Sp\cB$. Therefore, in order to ensure that the domain of the coarse assembly map is well-defined, we need that $E  \cO^{\infty}$ is a closed local homology theory, interpreted here as a colimit preserving functor $\Sp\cB\to \bC$, Corollary~\ref{giugoeggergg}.   For this reason we have to assume that $E$ is strong, because this is an assumption in  Lemma \ref{fewoiiofwefewf}.
\hB
\end{rem}

\begin{rem}
It is easy to see by inspecting the constructions  that the coarse assembly map is natural in the strong coarse homology theory $E$. If $E\to E'$ is a natural transformation between strong coarse homology theories, then
\[
\xymatrix{E\cO^{\infty}\bP\ar[r]^-{\mu_{E}}\ar[d]&\Sigma E\ar[d]\\E'\cO^{\infty}\bP\ar[r]^-{\mu_{E'}}&\Sigma E'}
\]
is a natural commuting square.
\hB
\end{rem}

\begin{rem}
By Lemma \ref{fewoiiofwefewf} we know that   $E\cO^{\infty}$ is a  closed and open local homology theory.
In view of Lemma \ref{vweoijiorjovbwvdfsv} we can replace $\bP$ by the open variant $\bP^{o}$ without changing the assembly map.
\hB
\end{rem}

\begin{rem}\label{fwwoiuweofewfewfw}
It follows from the above fibre sequence \eqref{gerpogikrepoergerg} that for a bornological coarse space~$X$ the coarse assembly map \begin{equation}\label{oioiejogiergerergreg}
\mu_{E,X} \colon  E\cO^{\infty}\bP(X)\to \Sigma E(X)
\end{equation}
is an equivalence if and only if
$E\cO \bP(X)\simeq 0$. Therefore we have identified
$E\cO\bP$ as the  coarse homology theory which detects the obstructions to $\mu_{E}$ being an  equivalence.
\hB
\end{rem}

\begin{rem}
At the moment the local homology theory $E\cO^{\infty}$ appearing in the domain of the coarse assembly map might appear mysterious. In Proposition~\ref{fifowefweewfwef} we calculate the   evaluation of this homology theory on finite-dimensional, locally finite simplicial complexes under the assumption that $E$ is   additive. 
%
\hB
\end{rem}

\begin{rem}
In the case of coarse $K$-homology there is the analytic coarse assembly map  \cite{hr}. 
It is only defined for $X$ in $\UBC$ presented by a proper metric space. 
It is a homomorphism from the locally finite $K$-homology groups
of $X$ to its coarse $K$-homology groups.  It is a non-trivial matter to compare our proposed version of  the assembly map with the one in   \cite{hr}. We will discuss this problem in
{\cite{compass}.}\footnote{This comparison is also considered in {Section~16 of the arXiv-preprint version~v2} of the present paper.}   

The  analytic coarse assembly map is closely related with index theory. In contrast, the coarse assembly map introduced
in Definition   \ref{fiwjfofewfewfwefefweffw} is of geometric and homotopy theoretic nature.
In contrast to the analytic coarse assembly map   it is  a natural transformation between coarse homology theories. This fact allows to   apply the comparison theorems  shown in \cite{buen}, see e.g.\ Theorem \ref{woifowfwewfewf}.
\hB
\end{rem}

\section{Isomorphism results}\label{ergop34t34t34t34}

In this section we  discuss conditions which imply that the coarse assembly map $\mu_{E,X}$ in \eqref{oioiejogiergerergreg} is an equivalence. We will discuss the cases of finite asymptotic dimension, finite decomposition complexity, and scaleable spaces. Our goal is to show that in many cases the  reasons for the  validity of the coarse Baum--Connes or Farrell--Jones conjectures for $X$ in fact imply in greater generality that the coarse assembly map $\mu_{E,X}$ is an equivalence for suitable coarse homology theories $E$.
 
Note that the coarse assembly map $\mu_{E} \colon E\cO^{\infty}\bP\to  \Sigma E$ (Definition \ref{fiwjfofewfewfwefefweffw}) is a morphism between coarse homology theories. So it is clear from the outset that the property of $\mu_{E,X} $ of  being an equivalence only depends on the coarse motivic spectrum $\Yo^{s}(X)$.

\subsection{{Finite asymptotic dimension}}

Let $X$ be a bornological coarse space with bornology $\cB$ and coarse structure $\cC$, see Section~\ref{eoijiqoegergegwegwerg}. Recall that $X$ is called discrete as a coarse space if 
$\cC$   is the minimal coarse structure $ \cC_{min}$ consisting of all subsets of the diagonal.

Let $X$ be a bornological coarse space, and let $E$ be a strong coarse homology theory.

\begin{prop}\label{prop_assembly_equiv_discrete}
If $X$ is discrete as a coarse space, then
the coarse assembly map $\mu_{E,X}$ is an equivalence.
\end{prop}

\begin{proof}
Since $X$ is discrete as a coarse space we have $X\cong P_{\diag}(X)$
as uniform bornological coarse spaces  if we equip $X$ with the discrete uniform structure.
By  \cite[Prop.~9.35]{equicoarse} the boundary map of the cone sequence \eqref{gerpogikrepoergerg} induces an equivalence
\[  \cO^{\infty}(X)\stackrel{\simeq}{\to} \Sigma \Yo^{s}(F_{\cU}(X))\, .\qedhere\]
\end{proof}

Let $\Sp\cX\langle \mathrm{disc}\rangle$ denote the cocomplete stable full subcategory of $\Sp\cX$  generated by the motives of all discrete bornological coarse spaces. Let $E$ be a strong coarse homology theory.
Proposition~\ref{prop_assembly_equiv_discrete} has the following immediate consequence.

\begin{kor}\label{fweoiowefewew}
The coarse assembly map $\mu_{X,E}$ is an equivalence for all bornological coarse spaces $X$ such that $\Yo^{s}(X) \in \Sp\cX\langle \mathrm{disc}\rangle$.
\end{kor}

Let $X$ be a  coarse space with coarse structure $\cC$.

\begin{ddd}\label{wetoijovfsdvsdfvsfdv}
$X$ has weakly finite asymptotic dimension if there exists a cofinal set of entourages $U$  in $\cC$ such that $X_{U}$ has finite asymptotic dimension.
\end{ddd}

Let $X$ be a bornological coarse space. We apply Definition \ref{wetoijovfsdvsdfvsfdv} to its underlying coarse space.
Let $E$ be a strong coarse homology theory.
\begin{theorem}\label{woifowfwewfewf}
 If $X$ has weakly finite asymptotic dimension, then
 the coarse assembly map $  \mu_{E,X}:E\cO^{\infty}\bP(X)\to \Sigma E(X)$ is an equivalence.
\end{theorem}

\begin{proof}
Since the coarse assembly map is a natural transformation between coarse homology theories, it extends to a natural transformation between colimit preserving functors  defined on $\Sp\cX$ (see \cite[Cor.~4.24 and Rem.~4.27 ]{buen} for a precise interpretation of this statement).
As a consequence of Corollary~\ref{fweoiowefewew} this extension is an equivalence on motives in  $\Sp\cX\langle \mathrm{disc}\rangle$.

The assumptions on the space $X$ imply by \cite[Thm.~5.59]{buen}
that the motive $\Yo^{s}(X)$ belongs to $\Sp\cX\langle \mathrm{disc}\rangle$. 
\end{proof}

In the following we use Corollary \ref{fweoiowefewew} in order to  show that the homology theory $E\cO^{\infty}\bP$ inherits some pleasant additional properties from $E$.

 Let $E:\BC\to \bC$ be a coarse homology theory and assume in addition that $\bC$ is complete. If $X$ is a set and $x$ is a point in $X$, then  we can define by excision a projection \begin{equation}\label{qewfkhuihuihuihiuj}
E(X_{min,min})\simeq E(\{x\})\oplus E(X_{min,min}\setminus \{x\})\to E(\{x\})\ .
\end{equation} 
The family of these projections for all $x$ in $X$ induces a map
\begin{equation}\label{adbfhjakfdaewfadf}
E(X_{min,min})\to \prod_{x\in X}E(\{x\})\, .
\end{equation}

\begin{ddd}[{\cite[Def.~6.4]{buen}}]\label{eroigjoerggerferwef}
$E$ is called additive if \eqref{adbfhjakfdaewfadf} is an equivalence for every set $X$.
\end{ddd}

Let $E$ be a strong coarse homology theory.

\begin{prop}
If $E$ is  additive, then $E\cO^{\infty}\bP$ is additive.
\end{prop}

\begin{proof}
Since $X_{min,min}$  and $\{x\}$ are discrete, the coarse assembly maps $\mu_{E,X_{min,min}}$ and $\mu_{E,\{x\}}$ are equivalences.
We have the following diagram  
\[\xymatrix{
E\cO^{\infty}\bP(X_{min,min})\ar[d]^{?}\ar[rr]^-{\mu_{E,X_{min,min}}}_-{\simeq}&& \ar[d]^{\simeq}_{\eqref{adbfhjakfdaewfadf} } \Sigma E(X_{min,min})\\\prod_{x\in X} E\cO^{\infty}\bP(*)\ar[rr]^-{\prod_{x\in X}\mu_{E,\{x\}}}_-{\simeq}&&\prod_{x\in X}\Sigma E(*)
}\]
which commmutes by the naturality of the assembly map. 
For the right vertical equivalence we use that $E$ is additive, and that $\Sigma$ preserves products by stability of $\bC$. We conclude that
the morphism marked by $?$ is an equivalence and hence the additivity of $E\cO^{\infty}\bP$.
 \end{proof}

In \cite[Lem.~2.25]{buen} we observed that $\BC$ admits (very small) coproducts. If $E \colon \BC\to \bC$ is a coarse homology theory, then implicitly
$\bC$ is cocomplete and therefore admits (very small) coproducts as well. We can consider the property that $E$ preserves coproducts \cite[Def.~6.9]{buen}.  
 
 Let $E$ be a strong coarse homology theory.

\begin{prop}
If $E$  preserves coproducts, then $E\cO^{\infty}\bP$ preserves coproducts.
\end{prop}
 
\begin{proof}
  Let $(X_{i})_{i\in I}$ be a family in $\BC$. 
In \cite[Lem.~4.12]{buen} it is shown that the fibre  $F$ of the canonical map
\begin{equation}
\label{eq_map_coproduct_fiber}
\bigoplus_{i \in I} \Yo^s(X_i) \to \Yo^s\big(\coprod_{i \in I} X_i\big)
\end{equation}
belongs to $\Sp\cX\langle \mathrm{disc}\rangle$. 
We apply the extension of the  assembly $\mu_{E}$ map to the corresponding fibre sequence and obtain the following commuting diagram:
$$\xymatrix{
E\cO^{\infty}\bP(F)\ar[d]^{\mu_{E,F}}_{\simeq}\ar[r]&\bigoplus_{i\in I}E\cO^{\infty}\bP(X_{i})\ar[r]^{?}\ar[d]^{\oplus_{i\in I}\mu_{E,X_{i}}}&E\cO^{\infty}\bP\big(\coprod_{i \in I} X_i\big)\ar[d]^{\mu_{E,\coprod_{i \in I} X_i}}\\
\Sigma E(F)\ar[r]&\bigoplus_{i\in I}\Sigma E(X_{i})\ar[r]^{\simeq}&\Sigma E\big(\coprod_{i \in I} X_i\big)
}$$
The left vertical morphism is an equivalence since $F$ belongs to  $\Sp\cX\langle \mathrm{disc}\rangle$.
This implies that the right square is a pull-back square. 
The lower right horizontal morphism is an equivalence since $E$ preserves coproducts. 
We conclude that the morphism marked by $?$ is an equivalence. Hence $E\cO^{\infty}\bP$ preserves coproducts.
%
%
%
%
%
%
%
%
\end{proof}

\begin{rem}
If $X$ is a set, then we have an isomorphism $X_{min,max}\cong \bigsqcup_{x\in X}\{x\}$, where the coproduct is taken in $\BC$. In particular, if $X$ is infinite, then the natural morphism $X_{min,max}\to X_{min,min}$ (given by the identity map of $X$) is not  an isomorphism.  Additivity 
of a coarse homology theory  
is  a condition which differs from  the condition of being
coproduct preserving.  A coarse homology theory $E$ may have both properties at the same time, and in this case
$E(X_{min,max})\to E(X_{min,min})$ is equivalent to the natural inclusion
$$\bigoplus_{x\in X} E(\{x\})\to \prod_{x\in X} E(\{x\})\, .$$
\end{rem}

\subsection{{Finite decomposition complexity}}

Guentner, Tessera and Yu \cite{Guentner:2010aa} introduced a weaker condition than finite asymptotic dimension called finite decomposition complexity (FDC). In \cite{transb} we investigated under which assumptions on $E$ the condition that a bornological coarse space $X$ has FDC implies that $E\cO\bP(X)\simeq 0$ (even in the equivariant case).

The results in  \cite{transb} require the additional assumptions  that $E$  is weakly additive and  admits transfers.  In the following we explain these conditions.
  
In \cite{coarsetrans} we introduced the category $\BC_{tr}$ of bornological coarse spaces with transfers. It is an enlargement of the category $\BC$ by adding transfer morphisms.

Given a set $I$ and a bornological coarse space we can form the bornological coarse space $I_{min,min}\otimes X$ (see \cite[Ex.~2.17]{equicoarse}). Let $j_{i}\colon X\to I_{min,min}\otimes X$ denote the inclusion of the component with index $i$ which is a morphism in $\BC$. By design $\BC_{tr}$ contains a transfer morphism 
\[
\tr_{X,I}\colon X\to  I_{min,min}\otimes X
\]
which morally is the sum $\sum_{i\in I} j_{i}$ of the inclusion morphisms.

If $E$ is a coarse homology theory, then the construction of an extension $E_{tr}$ to the category $\BC_{tr}$ should be guided by the idea that the morphism 
\[
E_{tr}(\tr_{X,I})\colon E_{tr}(X)\to E_{tr}( I_{min,min}\otimes X)
\]
places identical copies of a cycle for $E_{tr}(X)$ on each copy $\{i\}\times X$ of $X$ in  $I_{min,min}\otimes X$.

A coarse homology theory with transfers is a functor $E_{tr}:\BC_{tr}\to \bC$ such that its restriction $E \colon \BC\to \bC$   along the  inclusion $\BC\to  \BC_{tr}$ is a coarse homology theory.  We say that $E_{tr}$ extends $E$.

One can show that for every
 $i$ in $I$
 the composition
\begin{equation}\label{frfjiofj34f34f3f}
E_{tr}(X)\xrightarrow{E_{tr}(\tr_{X,I})} E_{tr}(I_{min,min}\otimes X)\stackrel{\text{excision}}{\simeq} E_{tr}(X)\oplus E_{tr}((I\setminus \{i\})_{min,min}\otimes X) 
\end{equation}
  is equivalent to 
$$\id_{E_{tr}(X)}\oplus E_{tr}(\tr_{X,I\setminus\{i\}})\, .$$
Let $E$ be a  coarse homology theory.
\begin{ddd}[{\cite[Def.~1.2]{coarsetrans}}]\label{qwfkjqlwdewdewdqdqed}
$E$  admits transfers if it has an extension $E_{tr}$ to a coarse cohomology theory with transfers.
\end{ddd}

A coarse homology theory $E \colon \BC\to \bC$  is called strongly additive \cite[Def.~3.12]{equicoarse} if $\bC$ admits products and $E$ sends free unions to products, i.e., if
$$E\big(\bigsqcup^{\free}_{i\in I} X_{i}\big) \xrightarrow{\simeq} \prod_{i\in I} E(X_{i})$$
for every family $(X_{i})_{i\in I}$ of bornological coarse spaces, where the map is induced by the family of projections
$(E(\bigsqcup^{\free}_{i\in I} X_{i})\to E(X_{i}))_{i\in I}$ given  by excision.
For the definition of weak additivity (appearing in the assumptions of Theorem \ref{feoijofwefewfewf} below) we refer to \cite[Def.~2.23]{transb}. Note that 
strong additivity implies weak additivity  and additivity  in the sense of Definition \ref{eroigjoerggerferwef}.


\begin{ex}\label{wrgfiowrgegergegergerg}
The  coarse homology theories listed in the Example \ref{wrgfiowrgegergegergerg111} are all 
 strongly additive and admit transfers:
\begin{enumerate}
\item Coarse ordinary homology  \cite{coarsetrans}.
\item Coarse algebraic $K$-homology with coefficients in an additive category  \cite{coarsetrans}.
\item Coarse Waldhausen $K$-homology of spaces   with coefficients in a space  \cite{Bunke:aa}.
\item Coarse algebraic $K$-homology with coefficients in a left-exact $\infty$-category  \cite{unik}.
\item Coarse topological $K$-homology with coefficients in a $C^{*}$-category   {\cite{bu}}.\qedhere
\end{enumerate}
\end{ex}


Let $X$ be a bornological coarse space, and let $E \colon \BC\to \bC$ be a strong coarse homology theory.
\begin{theorem}\label{feoijofwefewfewf}
Assume:
\begin{enumerate}
\item \label{fewiofhwiofewf} $\bC$ is compactly generated. 
\item \label{fewiofhwiofewf1} $E$ is weakly additive.
\item \label{fewiofhwiofewf2} $E$ admits transfers.
\item $X_U$ has FDC for a cofinal set of entourages $U$ of $X$. \end{enumerate} Then the coarse assembly map $\mu_{E,X}$ is an equivalence.
\end{theorem}

\begin{proof}
This  follows from \cite[Thm.~1.3]{transb} and Remark~\ref{fwwoiuweofewfewfw}.
\end{proof}

\begin{rem}
Finite asymptotic dimension implies FDC. {Therefore}
the above Theorem~\ref{feoijofwefewfewf} generalizes Theorem \ref{woifowfwewfewf} provided $E$ has the required additional properties.
\hB
\end{rem}

\subsection{Scaleable spaces}
  
In the literature on the coarse Baum--Connes conjecture it is an important observation that the existence of a suitable scaling implies that  the analytic coarse assembly map in coarse $K$-homology is an isomorphism  \cite{hr}.  In the following we show analogous results for general coarse homology theories.

 Let $X$ be a uniform bornological coarse space, and let $s \colon X\to X$ be a morphism of uniform bornological coarse spaces.
 We assume that the uniform structure of $X$ is induced by a metric.
 \begin{ddd}\label{eiowefewfewfewf}
 The morphism $s$ is a scaling if it satisfies the following conditions:
 \begin{enumerate}
 \item\label{4u8934u894343434} $s$ is $1$-Lipschitz.
\item\label{t3io3t34t43t43t43t43t3} For every coarse entourage $W$ and uniform entourage $V$ of $X$ there exists $k$ in $\nat$ such that
$(s^{k}\times s^{k})(W)\subseteq V$.
\item\label{efoief23f23f2f2f2ff} For every coarse entourage $U$ of $X$   the union $\bigcup_{k\in \nat} (s^{k}\times s^{k})(U)$ is also a coarse entourage of $X$.
 \end{enumerate}
\end{ddd}

\begin{ex}\label{dviowewfefwefewfwe}
Assume that $X$ is a proper metric space whose  structures are induced from the metric. If
$s \colon X\to X$ is a map which is  $1/2$-Lipschitz and proper, then $s$ is a scaling in the sense of Definition \ref{eiowefewfewfewf}.
Note that in order to be a scaling in the sense of \cite[Def.~7.1]{hr} one must in addition assume that $s$ is coarsely and properly  homotopic to the identity. These conditions will be added in Definition \ref{ewfijwoefwefewfew543534} which characterizes coarse scalings.
\hB
\end{ex}

Using the existence of a scaling for $X$ we want to deduce that $E\cO(X)\simeq 0$ for suitable coarse homology theories $E$.
Similarly as in  the proof of \cite[Thm.~7.2]{hr}   the argument is based on an Eilenberg swindle.
In order to make this work in our abstract setting we {need to} assume that the homology theory admits transfers (Definition \ref{qwfkjqlwdewdewdqdqed}).

Let $X$ be a uniform bornological coarse space, and let $s \colon X\to X$ be a morphism. Furthermore let $E$ be a coarse homology theory.
\begin{prop}\label{ergioerogergregreg}
Assume:
\begin{enumerate}
\item $s:X\to X$ is a scaling.
\item \label{rgfoirgoregege4} $E(F_{\cU}(s))\simeq \id_{E(F_{\cU}( X))}$. 
\item $E$ admits transfers.
\item\label{ifjweofiewjfewioewffewfefe} $E\cO^{\infty}(s)\simeq \id_{E\cO^{\infty}(X)}$.
\end{enumerate}
Then
$E\cO(X)\simeq 0$.
\end{prop}

Before starting with the proof of the above proposition, let us first prove the following statement. Recall Definition \ref{rgfporgergergereg} of the modified cone $\tilde \cO(X)$. 
Let $s \colon X\to X$ be a scaling.
We define the map of sets $$\Phi \colon 
\nat\times [0,\infty)\times X\to [0,\infty)\times X\, ,\quad 
\Phi(n,t,x) \coloneqq (n+t,s^{n}(x))\, .$$

\begin{lem}
The map $\Phi$ is a morphism of bornological coarse spaces
$$\Phi \colon  \nat_{min,min}\otimes \tilde \cO(X)\to\tilde  \cO(X)\, .$$
\end{lem}

\begin{proof}
First we show that $\Phi$ is proper.
Let $B$ be a bounded subset  in $X$ and $u$ be in $\nat$ and consider the bounded 
subset $[0,u]\times B$ in $\tilde \cO(X)$. Then $\Phi^{-1}([0,u]\times B)$ is contained in
$[0,u]\times [0,\infty)\times X$.
The restriction of $\Phi$ to $\{n\}\times \tilde \cO(X)$ is proper for every $n$ in $\nat$ since the maps 
$s^{n}:X\to X$  and   $t\mapsto n+t:[0,\infty)\to [0,\infty)$ are proper.
Therefore we can conclude that
$\Phi^{-1}([0,u]\times B)$ is bounded.

We now show that $\Phi$ is controlled. It is easy to check using \ref{eiowefewfewfewf}.\ref{efoief23f23f2f2f2ff}
and the fact that $t\mapsto n+t$ is  $1$-Lipschitz that $\Phi$ is a morphism of bornological coarse spaces
$$\nat_{min,min}\otimes F_{\cU}([0,\infty)\otimes X)\to F_{\cU}([0,\infty)\otimes X)\, .$$
 
Let $\psi \colon [0,\infty)\to [0,\infty)$ be a   function such that $\lim_{t\to\infty}\psi(t)=0$. For simplicity we can assume that $\psi$ is monotonously decreasing. It determines a function $\phi \colon  [0,\infty) \to \cT$ by $\phi(t) \coloneqq U_{\psi(t)}$ as used in Definition \ref{rgfporgergergereg}.\ref{fiewjioffwefwefwef}. Let $W$ be a coarse entourage of $X$ and
$V \coloneqq U_{r}\times W$ be a coarse entourage of $[0,\infty)\otimes X$ for $r$ in $(0,\infty)$.   Then we must show that
$$(\Phi\times \Phi)(\diag(\nat)\times (V\cap U_{\phi})) \subseteq  U_{\phi^{\prime}}$$ for $\phi^{\prime}(t) = U_{\psi^{\prime}(t)}$ for some function
$\psi^{\prime}$ having the same properties as $\psi$.  This boils down to the assertion that for  all $t$ in $[0,\infty)$
we have $d(s^{n}(x),s^{n}(y))\le \psi^{\prime}(t)$ for all 
$n$ in $\nat$ with $t\ge n$ and $(x,y)\in W$ with $d(x,y)\le \psi(t-n-r) $ (here we use the monotonicity of $\psi$). Here we set $\psi(t) \coloneqq \psi(0)$ for negative $t$.

We define the monotonously decreasing function
$$e \colon \nat\to [0,\infty]\, , \quad e(n) \coloneqq \sup\{d(s^{n}(x),s^{n}(y))\:|\:   (x,y)\in W\}\, .$$ By \ref{eiowefewfewfewf}.\ref{t3io3t34t43t43t43t43t3}
   we have
   $\lim_{n\to \infty} e(n)=0$.
 We define
$$\psi^{\prime}(t) \coloneqq  \max\{\min\{\psi(t-n-r), {e(n)}\} \:|\: n\in \nat\:\&\: t\ge n\}\, .   $$
In view of   \ref{eiowefewfewfewf}.\ref{4u8934u894343434} this function would do the job
if $\lim_{t\to\infty} \psi^{\prime}(t)=0$.
Let $\epsilon$ in $(0,\infty)$ be given. Then we choose $n_{0}$ in $\nat$ so large that
$e(n)\le \epsilon$ for all $n$ in $\nat$ with $n\ge n_{0}$. Let furthermore $t_{0}$ in $[0,\infty)$ be so large that $\psi(t)\le \epsilon$ for all $t$ in $[t_0,\infty)$. If $t$ in $[0,\infty)$ satisfies
$t\ge n_{0}+t_{0}+r$, then
$\psi^{\prime}(t)\le \epsilon$.
\end{proof}

\begin{proof}[{Proof of Proposition~\ref{ergioerogergregreg}}]
Let $E_{tr}$ be an extension of $E$ to a coarse homology theory with transfers.
An application of the relation \eqref{frfjiofj34f34f3f} yields   a decomposition
\begin{equation}
\label{eqfegr4t534weret4}
E_{tr} ( \Phi\circ \tr_{\tilde\cO(X),\nat})\simeq  E_{tr}( \Phi^{\prime}\circ \tr_{\tilde\cO(X),\nat^{\ge 1}})+  \id_{E(\tilde \cO(X))}\ ,
\end{equation}
where $\Phi^{\prime}$ is the restriction of $\Phi$ to $\nat^{\ge 1}_{min,min}\otimes \tilde \cO(X)$.  
We consider the following commuting diagram in  $\BC_{tr}$:
\[\xymatrix{
\tilde \cO(X)\ar[d]_-{\tilde \cO(s)}\ar[rr]^-{\tr_{\tilde \cO(X),\nat^{\ge 1}}} & & \nat^{\ge 1}_{min,min}\otimes \tilde \cO(X)\ar[d]^{(n\mapsto n-1) \otimes \tilde \cO(s)}\ar[rr]^-{\Phi^{\prime}} & & \tilde  \cO(X) \\
\tilde \cO(X)\ar[rr]^-{\tr_{\tilde \cO(X),\nat}} & & \nat_{min,min}\otimes \tilde  \cO(X)\ar[rr]^-{\Phi} & & \tilde  \cO(X)\ar[u]_-{T}
}\]
where $T \colon \tilde \cO(X)\to \tilde \cO(X)$ is given by $T(t,x) \coloneqq (t+1,x)$. Note that the morphism $T$ is close to the identity.
  
Since $T$ is close to the identity  the commutativity of the above diagram implies
\begin{eqnarray*}
E_{tr}( \Phi^{\prime}\circ \tr_{\tilde\cO(X),\nat^{\ge 1}})&\simeq&
E_{tr}(T)\circ E_{tr}( \Phi)\circ E_{tr}(\tr_{\tilde \cO(X),\nat})\circ E_{tr}\tilde \cO(s)\\
&\simeq& E_{tr}( \Phi\circ  \tr_{\tilde \cO(X),\nat})\circ E\tilde \cO(s)\, ,
\end{eqnarray*}
from which we get, {using \eqref{eqfegr4t534weret4},}
\begin{equation}\label{ewfoij23or23r23r23rr}
E_{tr} ( \Phi\circ \tr_{\tilde \cO(X),\nat})\simeq E_{tr}(\Phi\circ  \tr_{\tilde \cO(X),\nat})\circ E\tilde \cO(s) + \id_{E\tilde \cO(X)}\, .
\end{equation}

We now consider the diagram (note that we are now using the cone instead of the modified cone as above)
\begin{equation}\label{qfoihjoiqwefqewfq}\xymatrix{
E(F_{\cU}(X))\ar[rr]^-{\iota} \ar@/^0.5cm/[d]^{E(F_{\cU}(s))} \ar@/_0.5cm/[d] & & E\cO(X)\ar@{..>}@/^0.3cm/[dll]_-{\delta} \ar[rr] \ar@/^0.5cm/[d]^{E\cO(s)} \ar@/_0.5cm/[d] & & E\cO^{\infty}(X)\ar[rr] \ar@/^0.5cm/[d]^{E\cO^{\infty}(s)} \ar@/_0.5cm/[d] & & \Sigma E(F_{\cU}(X))\ar@/^0.5cm/[d]^{\Sigma E(F_{\cU}(s))}\ar@/_0.5cm/[d]\\
E(F_{\cU}(X))\ar[rr]^-{\iota} & & E\cO(X)\ar[rr] & & E\cO^{\infty}(X)\ar[rr] & & \Sigma E(F_{\cU}(X))
}
\end{equation}
 
whose horizontal sequences are two copies of the cone sequence and the non-labeled vertical maps are induced by the identity. 
The diagram is a picture of two morphisms between fibre sequences (one is the identity) which we want to compare.
The Condition~\ref{ifjweofiewjfewioewffewfefe} yields    a morphism
$\delta \colon E\cO(X)\to E(F_{\cU}(X))$ such that
\begin{equation}
\label{eq243trerwert}
E\cO(s)-\id_{E\cO(X)}\simeq \iota\circ \delta\, .\footnote{This is an assertion about the lower triangle  in the left square in \eqref{qfoihjoiqwefqewfq}. We do not make any assertion about the upper triangle. 
}
\end{equation}
We then have 
\begin{equation}\label{ewfiuhi4r343}
\iota\circ  \delta\circ \iota\stackrel{\eqref{eq243trerwert}}{ \simeq} (E\cO(s)-\id_{E\cO(X)})\circ \iota
\stackrel{!}{\simeq} \iota \circ (E(F_{\cU}(s))-\id_{E(F_{\cU}(X))})\stackrel{}{\simeq}
0\, .
\end{equation}
where the equivalence marked by $!$ follows from the commutativity of the left squares in  \eqref{qfoihjoiqwefqewfq}, and the last equivalence is a consequence of 
Condition \ref{rgfoirgoregege4}.

In view of Lemma \ref{ewfiowefwefergergergerg} we get the same relations if we replace the cone by the modified cone. The equivalence \eqref{ewfoij23or23r23r23rr} implies that {(using in the second line \eqref{eq243trerwert} for the modified cone)}
\begin{eqnarray}
\id_{E\tilde \cO(X)} & \simeq & E_{tr}(\Phi\circ \tr_{\tilde \cO(X),\nat}) - E_{tr}(\Phi\circ \tr_{\tilde \cO(X),\nat})\circ E\tilde\cO(s)\nonumber\\
&\simeq & E_{tr}(\Phi\circ \tr_{\tilde \cO(X),\nat}) - E_{tr}(\Phi\circ  \tr_{\tilde \cO(X),\nat})\circ ({\iota \circ \delta + \id_{E\tilde \cO(X)}})\nonumber\\
& \simeq & - E_{tr}(\Phi\circ \tr_{\tilde\cO(X),\nat}) \circ \iota\circ  \delta
\label{ioujiodjewofefewfqw}\end{eqnarray}
If we {compose this equivalence from the right} with $\iota$ and use \eqref{ewfiuhi4r343}, then we get
$$\iota\simeq - E_{tr}(\Phi\circ \tr_{\nat,\tilde \cO(X)})  \circ \iota\circ  \delta\circ \iota \simeq 0\, .$$
From \eqref{ioujiodjewofefewfqw} we conclude that 
$$ \id_{E\tilde \cO(X)}\simeq 0\ ,$$
which in view of Lemma \ref{ewfiowefwefergergergerg} implies $E \cO(X)\simeq 0$.
\end{proof}

%
%

Our next concern are  the conditions \ref{ergioerogergregreg}.\ref{rgfoirgoregege4} and \ref{ergioerogergregreg}.\ref{ifjweofiewjfewioewffewfefe}. Condition \ref{ergioerogergregreg}.\ref{rgfoirgoregege4} is satisfied, e.g., if $F_{\cU}(s)$
is coarsely homotopic to the identity map. In the literature this  is a standard assumption on a scaling; see, e.g., Higson--Roe \cite{hr}.

Condition \ref{ergioerogergregreg}.\ref{ifjweofiewjfewioewffewfefe} is more problematic.
If $s$ is homotopic to the identity in the sense of $\UBC$,  then  \ref{ergioerogergregreg}.\ref{ifjweofiewjfewioewffewfefe} is satisfied  by the homotopy invariance of the functor $E\cO^{\infty}\colon\UBC\to \bC$, 
Unfortunately, in applications $s$ is rarely homotopic to the identity in the sense of $\UBC$. The standard assumption made  e.g.\ in Higson--Roe \cite{hr} is that $F_{\cC,\cU/2}(s)$ is homotopic to the identity map, i.e., that $s$ is homotopic to the identity in the sense of $\TopBorn$ (i.e., after forgetting the coarse and the uniform structures, {but the homotopies are still required to be proper}). If $E$ is   additive, then  $E\cO^{\infty}$ has better homotopy invariance properties on nice spaces {which we will use in the following to make the standard assumption of Higson--Roe also work in our situation.}

 Let $X$ be a 
 uniform bornological coarse space, and let $E \colon \BC\to \bC$ be a 
 coarse homology theory. Note that this implicitly implies that
 $\bC$ is stable and cocomplete.

\begin{lem}\label{giwerogerggergregeg}
Assume:
\begin{enumerate}
\item \label{rgfoirgoregege1} $X$  is homotopy   equivalent (in $\UBC$) to  a  locally finite, finite-dimensional simplicial complex equipped with the metric structures.
 \item $\bC$ is complete.
 \item \label{rgfoirgoregege2} $E$ is  additive.
\item $F_{\cC,\cU/2}(s)$ is homotopic to $\id_{F_{\cC,\cU/2}(X)}$.
\end{enumerate}
Then
$E\cO^{\infty}(s) \simeq \id_{E\cO^{\infty}(X)}$.
\end{lem}

\begin{proof}
 This is an immediate consequence of Corollary \ref{rgigoeri34345345345} which will be shown below.
\end{proof}


In the following definition of a coarse scaling we introduce a class of scalings with additional properties ensuring that Proposition \ref{ergioerogergregreg} is applicable.

Let $X$ be a uniform bornological coarse space whose uniform structure is induced by a metric, and let $s \colon X\to X$ be a scaling.
\begin{ddd}\label{ewfijwoefwefewfew543534}
The scaling $s$ is a coarse scaling if it satisfies in addition:
\begin{enumerate}
\item $F_{\cU}(s)$ is coarsely homotopic to the identity.
\item $F_{\cC,\cU/2}(s)$ is properly homotopic to the identity.
\end{enumerate}
\end{ddd}

\begin{rem}
A scaling in the sense of \cite[Def.~7.1]{hr} is a coarse scaling; see also Example \ref{dviowewfefwefewfwe}.
\hB
\end{rem}

The following corollary is an analog of Higson--Roe \cite[Thm.~7.2]{hr}. Assumption \ref{wrefiowefwfewfewfewfw444}.\ref{ihgioergregregerger} {does not occur} in \cite{hr} because 
the analogue of our $E\cO^{\infty}$ is the functor  $X\mapsto K_{*}(D^{*}(X))$ in the notation of \cite{hr}
which has good homotopy invariance properties replacing the application of our Lemma \ref{giwerogerggergregeg}.

\begin{kor}\label{wrefiowefwfewfewfewfw444}
Assume:
\begin{enumerate}
\item $\bC$ is complete.
\item $E$ is   additive and admits transfers.
 \item The uniform structure of $X$ is induced by a metric.
\item  $X$  is homotopy   equivalent (in $\UBC$) to  a locally finite, finite-dimensional simplicial complex equipped with the metric structures.
\item $X$ admits a coarse scaling (see Definition \ref{ewfijwoefwefewfew543534}).
\end{enumerate}
Then $E\cO (X)\simeq 0$ and the cone boundary $E\cO^{\infty}(X)\to \Sigma E(X)$ is an equivalence.
\end{kor}
\begin{proof} This follows from  Proposition \ref{ergioerogergregreg}. Lemma \ref{giwerogerggergregeg} verifies Assumption \ref{ergioerogergregreg}.\ref{ifjweofiewjfewioewffewfefe}.
\end{proof}

\begin{ex}\label{ihfiweofwefewfe2}
A typical example of a uniform bornological coarse space which admits a coarse scaling is {a Euclidean} cone. Let $Y$ be a subset of the unit sphere in a Hilbert space, and let $X$ be the cone over $Y$ with the metric induced from the Hilbert space. 
We consider $X$ as a uniform bornological coarse space with all structures induced from the metric.
Then the map
$$s \colon X\to X\, , \quad s(x) \coloneqq x / 2$$
is a coarse scaling.

If $Y$ has a finite-dimensional, locally finite triangulation with a uniform bound on the size of its simplices, then
so does $X$.  In this case Corollary \ref{wrefiowefwfewfewfewfw444} can be applied to $X$.
\hB
\end{ex}


Let $X$ be a uniform bornological coarse space, and let $E \colon \BC\to \bC$ be a strong 
coarse homology theory.

\begin{theorem}\label{wrefiowefwfewfewfewfw}
Assume:
\begin{enumerate}
 \item \label{fwoifeowfwef24} $\bC$ is  complete.
\item\label{ergioergergerg} $E$ is   additive and admits transfers.
\item The uniform structure of $X$ is induced by a metric.
\item\label{ihgioergregregerger}   $X$  is homotopy   equivalent (in $\UBC$) to  a   locally finite, finite-dimensional simplicial complex equipped with the metric structures.
\item $X$ admits a coarse scaling (see Definition \ref{ewfijwoefwefewfew543534}).
\item\label{gioergjoierg34t334t55} $X$ is coarsifying (Definition \ref{wefiowef34t6453}).
\end{enumerate}
Then $E\cO\bP(F_{\cU}(X))\simeq 0$ and therefore the coarse assembly map $\mu_{E,F_{\cU}(X)}$ is
an equivalence.
\end{theorem}

\begin{proof}
Since $X$ is coarsifying and $E\cO$ is a local homology theory (Lemma \ref{fewoiiofwefewf}) we have an equivalence
$E\cO\bP(F_{\cU}(X))\simeq E\cO(X)$. We now apply Corollary \ref{wrefiowefwfewfewfewfw444} in order to conclude that
$E\cO(X)\simeq 0$.
\end{proof}

\begin{ex}\label{ex2435trfwer}
Let $Y$ and $X$ be as in Example \ref{ihfiweofwefewfe2}.
In general we can not expect $X$ to be coarsifying even if $Y$ is compact and the Hilbert space is finite-dimensional.
Especially, we do not {know if}
the analogue of \cite[Prop.~4.3]{hr} is true in our generality. By using Proposition \ref{prop34redsfg} one can prove that $X$ is coarsifying if $Y$ is a finite simplicial complex. Hence one can apply Theorem \ref{wrefiowefwfewfewfewfw} to Euclidean cones over finite complexes.

Therefore we get the analogue of \cite[Cor.~7.3]{hr} under the additional assumption of $Y$ being a finite simplicial complex
(instead of a finite-dimensional compact metric space).

Every complete, simply-connected, non-positively curved Riemannian manifold is coarsely homotopy equivalent
to the Euclidean cone over a finite-dimensional sphere. Since a  finite-dimensional sphere has a finite triangulation, Theorem \ref{wrefiowefwfewfewfewfw} provides a generalization of \cite[Cor.~7.4]{hr}.

Because  of the Assumptions \ref{wrefiowefwfewfewfewfw}.\ref{ihgioergregregerger} and \ref{wrefiowefwfewfewfewfw}.\ref{gioergjoierg34t334t55} we are not able to apply Theorem \ref{wrefiowefwfewfewfewfw} to cones over arbitrary compact metric spaces. In particular, we do not obtain the analogue of \cite[Cor.~8.2]{hr} asserting the coarse Baum--Connes conjecture for all hyperbolic (proper) metric spaces.

We do not know whether we should expect that the assmbly map $\mu_{E,F_{\cU}(X)}$ is an equivalence for all hyperbolic (proper) metric spaces or Euclidean cones over finite-dimensional compact metric spaces and arbitrary coarse homology theories $E$ satisfying the Assumptions  \ref{wrefiowefwfewfewfewfw}.\ref{ergioergergerg} and \ref{wrefiowefwfewfewfewfw}.\ref{fwoifeowfwef24}.
\hB
\end{ex}

The next corollary specializes Theorem \ref{wrefiowefwfewfewfewfw} by utilizing a convenient condition on the space $X$ to be coarsifying.
Let $E \colon \BC\to \bC$ be a  strong 
coarse homology theory.
Let $K$ be a simplicial complex, and let $K_{d}$ be the associated uniform bornological coarse space. 
\begin{kor}\label{wrfiowefwewfewfw}
Assume:
\begin{enumerate}
\item $\bC$ is complete.
\item $E$ is  additive  and admits transfers.
\item $K$ has bounded geometry.
\item $K_{d}$ is {equicontinuously} contractible.
\item $K_{d}$ admits a coarse scaling.
\end{enumerate}
 Then the coarse assembly map 
$\mu_{E,F_{\cU}(K_{d})}$ is an equivalence.
\end{kor}
\begin{proof}
Combine Proposition \ref{prop34redsfg} with Proposition \ref{wrefiowefwfewfewfewfw}.
\end{proof}

\begin{ex}
If $X$ is a tree or an affine Bruhat--Tits building  of bounded geometry, then Corollary \ref{wrfiowefwewfewfw} applies to $X$.
Hence we obtain the analogue of \cite[Cor.~7.5]{hr}.
\hB
\end{ex}

\begin{ex}
All  coarse homology theories listed in Example~\ref{wrgfiowrgegergegergerg} are strong,  additive, and admit transfers. Moreover, their target categories are complete.
Therefore the above theorems  apply to them.
 
Examples of spaces admitting coarse scalings and which are homotopy equivalent (in $\UBC$) to uniformly contractible simplicial complexes of bounded geometry are simply-connected complete Riemannian manifolds $M$ with sectional curvatures satisfying $-C \le \mathrm{sec} \le 0$ for a positive constant $C$. The coarse scaling $s$ is in this case given by, e.g., $s(x)  \coloneqq  \exp(\log(x)/2)$, where we have fixed a base point $x_0$  in $M$, $\exp  \colon  T_{x_0} M \to M$ is the Riemannian exponential map, and $\log  \colon  M \to T_{x_0} M$ is its inverse.
\hB
\end{ex}

\section{Calculation of \texorpdfstring{$\boldsymbol{E\cO^{\infty}}$}{EOinfty}}

The goal of this section is to provide a computation of $E\cO^{\infty}(X)$ in terms of the value  $E(\ast)$ of $E$ at the one-point space (see Proposition~\ref{fifowefweewfwef}). For this calculation we must 
 require that $E$ is  additive 
 



In this section we assume that $\bC$ is a stable and complete $\infty$-category. For the moment this suffices to construct the locally finite evaluation.
 Later we will in addition assume that $\bC$ is  cocomplete.

Let $F\colon\UBC \to \bC$ be a functor and $X$ a  small uniform bornological coarse space. 
\begin{ddd}\label{fwjeofiewffewwfewfw}
We define the locally finite evaluation of $F$ at   $X$  by
\begin{equation}\label{fwerfoio23r23r}
F^{\lf}(X):= {\lim}_{W} \ \Cofib( F(X\setminus W)\to F(X))\ ,
\end{equation}
where $W$ runs over all bounded subsets of $X$.
\end{ddd}
 
Similary as in  \cite[Rem.~7.16]{buen} one can turn {the above definition} into a construction of a functor $F^{\lf} \colon \UBC\to \bC$. 
\begin{rem}
Here are the details. We consider the category $\UBC^{\cB}$ of pairs $(X,W)$, where
$X$ is in $\UBC$ and $W$ is a bounded subset of $X$. A morphism $f \colon (X,W)\to (X^\prime,W^\prime)$ is a morphism $f \colon X\to X^{\prime}$ in $\UBC $ with $f(W)\subseteq W^{\prime}$. We have  the functors $$p \colon  \UBC^{\cB}\to \UBC \, , \quad p(X,W)  \coloneqq  X$$
and
$$\tilde F \colon \UBC^{ \cB} \to \bC\, , \quad \tilde F(X,W) \coloneqq  \Cofib( F(X\setminus W)\to F(X))\, .$$ We then define
the functor $F^{\lf}$ as the right Kan extension of $\tilde F$ along $p$:
$$\xymatrix{\UBC^{\cB}\ar[rr]^-{\tilde F}\ar[d]_-{p}&&\bC\,.\\\UBC \ar[urr]_-{F^{\lf}}}$$
This right Kan extension exists by our assumption on $\bC$, and the formula \eqref{fwerfoio23r23r}  follows from the pointwise formula for the evaluation of the right Kan extension.
\hB
\end{rem} 

\begin{rem}\label{efiowefwefewfefew}
If $F$ is induced from a functor $F^{\prime} \colon \TopBorn\to \bC$  by
$F=F^{\prime}\circ F_{\cC,\cU/2}$, then 
 we have an equivalence
$$F^{\lf}\simeq F^{\prime,\lf}\circ F_{\cC,\cU/2}\, ,$$
where
$F^{\prime,\lf}$ is exactly the locally finite evaluation as defined in \cite[Def.~7.15]{buen}. 
\hB
\end{rem}

We have a natural morphism $F(X)\to F^{\lf}(X)$.

\begin{lem}\label{wergiowergerwfrwfrefwfer}
If $F$ is homotopy invariant, then so is  $F^{\lf}$.
\end{lem}

\begin{proof}
The proof is the same as the one of 
\cite[Lem.~7.35]{buen}. One has just to observe that the subsets of the form
$[0,1]\times W$ of $[0,1]\otimes X$ are cofinal in the bounded  subsets of $[0,1]\otimes X$.
\end{proof}

\begin{lem}\label{riguhierofewewfqfewf}
If $F$ satisfies excision for decompositions into closed or open subsets, then so does $F^{\lf}$.
\end{lem}

\begin{proof}
The argument is the same as for \cite[Lem.~7.36]{buen}.
\end{proof}

\begin{rem}
Note that open or closed excision in the sense of Definition \ref{reffoiweffewff} involves additional assumptions on the subsets.
If $F$ satisfies excision in the sense of this Definition, then it is not clear what kind of excision properties $F^{\lf}$ has. The problem is that  the intersection with $X\setminus W$ does not necessarily preserve coarsely or uniformly excisive pairs.

In Lemma~\ref{eoigepggrgregegerg} below we show that under the additional assumption that $F$ is invariant under coarsening the functor $F^{\lf}$ at least satisfies excision for decompositions of simplicial complexes into closed subsets.
\hB
\end{rem}

Let $X$ be a   uniform bornological coarse space.
\begin{ddd} A coarsening  $X^{\prime}$ of $X$ is a    \ubs obtained from $X$ by replacing the coarse structure by a larger one which is still compatible with the bornology.
\end{ddd}
Note that the identity of the underlying sets is a morphism $X\to X^{\prime}$ of uniform bornological coarse spaces.

Let $F \colon \UBC\to \bC$ be a functor.
\begin{ddd}\label{gfiorjgoergre345}We say that $F$ is invariant under coarsening if for every  \ubs $X$ and coarsening $X\to X^{\prime}$ the induced morphism 
 $F(X)\to F(X^{\prime})$ is an equivalence. 
\end{ddd}

\begin{ex}\label{fiuiowfewfewf}
The functor $\cO^{\infty} \colon \UBC\to \Sp\cX$ is invariant under coarsening, see \cite[Prop.~9.33]{equicoarse}.
\hB
\end{ex}

\begin{lem}\label{lem_locally_finite_invariant_coarsening}
If $F$ is invariant under coarsening, then so is $F^{\lf}$.
\end{lem}

\begin{proof}
The assertion follows immediately from the defining formula~\eqref{fwerfoio23r23r}.
\end{proof}

A simplicial complex has a spherical path quasi metric which induces a  metric uniform, a metric coarse, and a metric bornological structure.  If we consider a simplicial complex as an object of $\UBC$ equipped with these structures, then we say that it has the metric structures.

Let $X$ be a   simplicial complex with the metric uniform and coarse structures, a compatible bornology (not necessarily the metric one), and with a decomposition $(A,B)$ into closed subsets.

\begin{lem}\label{eoigepggrgregegerg}
If $F$   satisfies closed excision in the sense   of Definition \ref{reffoiweffewff} and is invariant under coarsening, then we have a push-out square
\begin{equation}\label{xqwxiuhiuxqx}
\xymatrix{F^{\lf}(A\cap B)\ar[r]\ar[d]&F^{\lf}(A)\ar[d]\\F^{\lf}(B)\ar[r]&F^{\lf}(X)}
\end{equation}
\end{lem}

\begin{proof}
We use that the cofibre of a map of cocartesian squares is a cocartesian square.

In the  limit \eqref{fwerfoio23r23r} we can restrict $W$ to run only over the interiors of subcomplexes. Then $X\setminus W$ is again a simplicial complex and
$(A\setminus W, B\setminus W)$ is  a decomposition of $X\setminus W$ into closed {subsets}. 

Note that in the terms $F(X\setminus W)$ in 
\eqref{fwerfoio23r23r} we must equip the set $X\setminus W$ with the uniform bornological coarse structures induced from $X$. 
The uniform  structure on $X\setminus W$ is also induced from the path-metric of $X\setminus W$, but this is in general not true for the coarse structure. Since $F$ is invariant under coarsening, we can, without changing the  value of $F$ on the spaces $X\setminus W$,  equip these spaces with the smaller coarse structures associated to  the intrinsic path metrics.

Using Example \ref{wfiofwefewfwefewf} we now see that excisiveness
of $F$ in the sense of Definition \ref{reffoiweffewff} can be applied to
the decompositions  $(A\setminus W,B\setminus W)$    of the complexes  $X\setminus W$ occuring in the limit \eqref{fwerfoio23r23r}.  We therefore have expressed the square \eqref{xqwxiuhiuxqx}   as a limit of  cofibres of maps of cocartesian squares, i.e., as a limit of cocartesian squares.
Since $\bC$ is stable, cartesian and cocartesian squares in $\bC$ are the same. 
Hence  \eqref{xqwxiuhiuxqx}  itself is a  cocartesian square.
\end{proof}

\begin{rem}\label{rem_open_excision_subcomplexes}
If $F$ is homotopy invariant, is invariant under coarsening,  and satisfies open excision  in the sense   of Definition \ref{reffoiweffewff},  then  a modified argument shows that also in this case  the square \eqref{xqwxiuhiuxqx} is a push-out provided we restrict to decompositions $(A,B)$ of $X$ into subcomplexes.
\hB
\end{rem}

Let $F \colon \UBC\to \bC$ be a functor and assume that $F$ is excisive (in any of the senses discussed above).
If $X$ is a  \ubs {with the discrete uniform and coarse structures} and $x$ is a point in $X$, then analogously as in  \eqref{qewfkhuihuihuihiuj} we have a natural projection morphism $F(X)\to F(\{x\})$.

\begin{ddd}\label{greoiiogergerger}
$F$ is called additive if for every \ubs $X$ with the discrete uniform and coarse structures and the minimal bornology the natural morphism
$$F(X)\to \prod_{x\in X} F(\{x\})$$
induced by the projections is an equivalence.
\end{ddd}

 The product exists by our standing assumption for $\bC$ in this section.

{Let $F\colon \UBC \to \bC$ be a functor.
\begin{lem}\label{lem_locally_finite_excisive}
If $F$ is excisive (in any of the senses discussed above), then $F^{\lf}$ is additive.
\end{lem}
\begin{proof}The argument is 
completely analogous to the proof of \cite[Lem.~7.30]{buen}.
\end{proof}
}

Let $F:\UBC\to \bC$ be a functor.
\begin{lem}\label{feuhweifwfewwf}
Assume:
\begin{enumerate}
\item\label{roifjfoerfwerfwrefwerfw} $F$ satisfies closed excision in the sense   of Definition \ref{reffoiweffewff}.
\item $F$ is homotopy invariant.
\item $F$ invariant under coarsening.
\item  $F$ is   additive.
\end{enumerate}
Then for every   locally finite, finite-dimensional simplicial complex $X$ equipped with the metric  structures the natural morphism
$F(X)\to F^{\lf}(X)$ is an equivalence.  
\end{lem}

\begin{proof}
We argue by a finite induction over the dimension.

Assume that $X$ is zero-dimensional. Then $X$ is discrete  as a uniform and coarse space and has the minimal bornology. 
The functor $F$ is additive by assumption, and the functor $F^{\lf}$ is additive by {Lemma~\ref{lem_locally_finite_excisive}.} 
Since  $F(*)\stackrel{\simeq}{\to} F^{\lf}(*)$ we can conclude that $F(X)\stackrel{\simeq}{\to} F^{\lf}(X)$.

In the higher-dimensional case we use that local finiteness of $X$ implies that its bornology is generated by the finite subcomplexes.
 
 Assume now that the assertion is true for complexes of dimension $n-1$.
If the complex $X$ is $n$-dimensional, then we can decompose $X$ into a closed tubular neighbourhood $Y$ of thickness $1/3$ of its $(n-1)$-skeleton
and a disjoint union $Z$ of $n$-simplices of size $2/3$ (see the picture on Page~\pageref{fig543er}). The intersection is then a disjoint union of tubular neighbourhoods of thickness $1/3$ of the boundaries of simplices of size $2/3$.

\begin{figure}[ht]
\centering\includegraphics[scale=0.40]{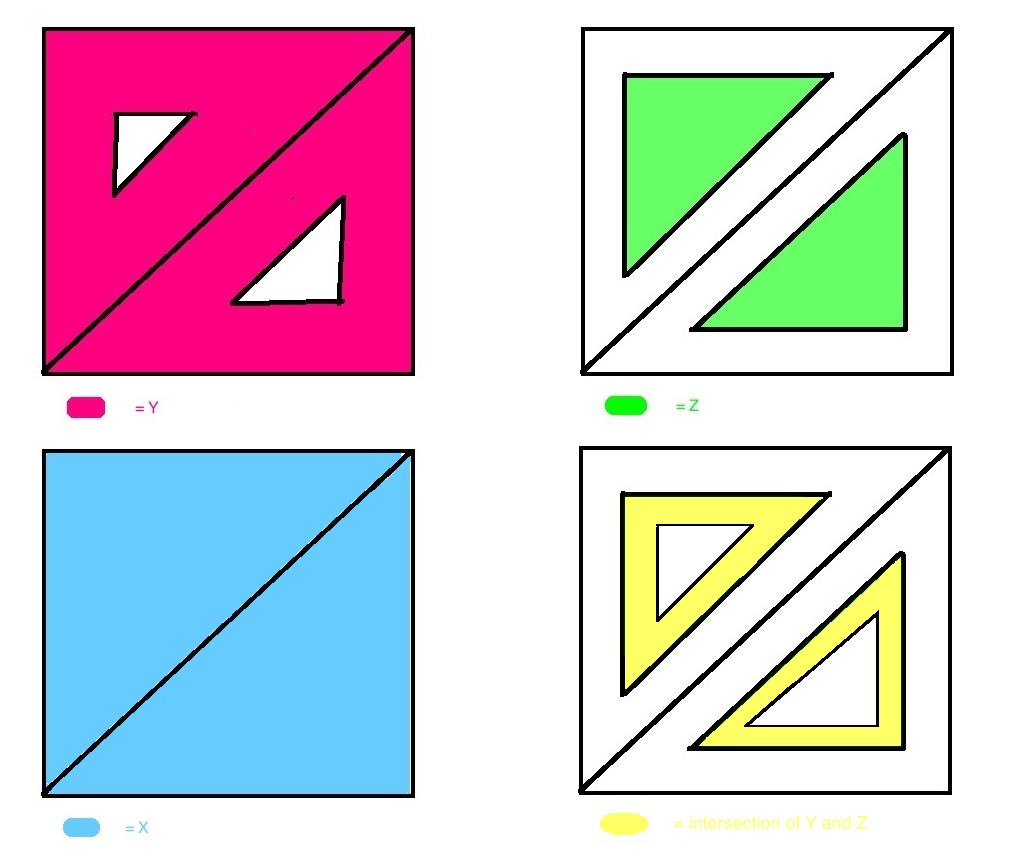}
\caption*{Decomposition $X=Y\cup Z$ used in the proof of Lemma~\ref{feuhweifwfewwf}}
\label{fig543er}
\end{figure}

This closed decomposition is coarsely and uniformly excisive. Hence we can apply excision for $F$ in the sense of Definition \ref{reffoiweffewff}. For $F^{\lf}$ we use Lemma \ref{eoigepggrgregegerg}.

We use homotopy invariance in order to replace the evaluation on {$Y$} by the evaluation on the $(n-1)$-skeleton $X^{n-1}$ itself. Furthermore, we can contract the $n$-simplices {of size $2/3$ in $Z$} to the set $C$ of their centers. Finally, we contract {$Y \cap Z$} to the {set $W$ of the boundaries of these simplices of size $2/3$.}

We use invariance under coarsening (note that $F^{\lf}$ is also invariant under coarsening {by Lemma~\ref{lem_locally_finite_invariant_coarsening}}) in order to replace the induced coarse structures by the coarse structures induced by the intrinsic path-quasi-metric on the $(n-1)$-skeleton $X^{n-1}$ and on $W$ and the discrete coarse structure on the set $C$ of centers of $n$-simplices. The bornology on $C$ induced from $X$ is the minimal one.

Then we can apply the induction assumption
to $X^{n-1}$, $W$ (which is also $n-1$-dimensional) and $C$.
\end{proof}

\begin{rem}\label{jknerinwefsd}
In Assumption~\ref{roifjfoerfwerfwrefwerfw} of Lemma~\ref{feuhweifwfewwf} one could replace ``closed'' by ``open'' without changing the conclusion. The argument must be slightly modified {by using the corresponding open versions of the decompositions in the induction steps.}
\hB
\end{rem}

%

Let $E \colon \BC\to \bC$ be a coarse homology theory such that $\bC$ is complete.
Recall Definition~\ref{eroigjoerggerferwef} of the notion of additivity and 
note that additivity follows from strong additivity \cite[Def.~3.12]{equicoarse}.

%


\begin{prop}\label{fiuwofewfewfewf}
If $E$ is additive and $X$ is a   locally finite, finite-dimensional  simplicial complex equipped with the metric structures, then the natural morphism
$$E\cO^{\infty}(X)\to (E\cO^{\infty})^{\lf}(X)$$
is an equivalence.
\end{prop}

\begin{proof}
We will check that the assumptions of Lemma~\ref{feuhweifwfewwf} are satisfied. 
By Lemma \ref{foprgfregegr} the functor $\cO^{\infty} $ satisfies excision  in the sense   of Definition \ref{reffoiweffewff} and is homotopy invariant. Hence $E\cO^{\infty}$ has these properties. Further, by Example \ref{fiuiowfewfewf} the functor $E\cO^{\infty}$ is invariant under coarsening.
 
Let $X$ be a uniform bornological coarse space {which is discrete both as a uniform and as a coarse space.} Then
$$\cO^{\infty}(X)\simeq \Sigma \Yo^{s}(F_{\cU}(X))$$ 
by \cite[Prop.~9.33]{equicoarse}. Using that $E$ is  additive at the marked equivalence {in the following chain of equivalences, we have for a  \ubs $X$ with the discrete uniform and coarse structures and the minimal bornology 
}
\begin{eqnarray*}
E\cO^{\infty}(X) & \simeq& E( \Sigma F_{\cU}(X))
  \simeq  \Sigma E(F_{\cU}(X)) 
  \stackrel{!}{\simeq}  \Sigma\big(\prod_{x\in X} E(\{x\})\big)\\
& \simeq& \prod_{x\in X} \Sigma E(\{x\}) 
 \simeq  \prod_{x\in X}E\cO^{\infty}(\{x\})
\end{eqnarray*}
showing that $E\cO^{\infty}$ is additive.  
\end{proof}

Following Weiss--Williams \cite{ww_pro}, for a homotopy invariant functor $F:\UBC \to \bC$ we can construct a  best approximation of $F$ by a homology theory. It is given by the Kan extension procedure described in  the proof of \cite[Prop.~7.43]{buen} which produces a functor and a natural transformation
$$F^{\%}\colon \UBC\to \bC\, , \quad F^{\%}\to F\, .$$ 
Here are the details. We assume that $\bC$ is  stable, complete and cocomplete.
Let $\UBC^{\Delta}$ be the category of pairs $(X,\sigma)$, where $X$ is in $\UBC$ and $\sigma:\Delta^{n}\to X$ is a continuous map.
On $\Delta^{n}$ we consider the uniform bornological coarse structure induced by the spherical metric. Then $\sigma$ is automatically a morphism in $\UBC$. A morphism $(X,\sigma)\to (X',\sigma')$ is a commutative diagram
$$\xymatrix{\Delta^{n}\ar[r]^{\phi}\ar[d]^{\sigma}&\Delta^{n'}\ar[d]^{\sigma'}\\X\ar[r]^{f}&X'}$$
where $f$ is a morphism in $\UBC$ and $\phi$ is induced by a morphism $[n]\to [n']$ in the category $\mathbf{\Delta}$.
We have functors $$p,q \colon \UBC^{\Delta}\to \UBC\, , \quad\quad  p(X,\sigma) \coloneqq  X\, , \quad q(X,\sigma) \coloneqq \Delta^{n}\, .$$
\begin{ddd} Then we define $F^{\%}$ by a left Kan-extension of $F\circ q$ along $p$:
\[
\xymatrix{\UBC^{\Delta}\ar[rr]^{F\circ q}\ar[d]^{p}&&\bC\,.\\
\UBC\ar[urr]^{F^{\%}}&}
\]
\end{ddd}

The objectwise formula for the left Kan extension yields the following formula for the values of $F^{\%}$:
$$F^{\%}(X)  \coloneqq  \colim_{(\Delta^n \to X)} F(\Delta^n)\, ,$$
where the colimit runs over the category of simplices of $X$. 

Since $\bC$ is stable and cocomplete,    it is tensored over the small category $\Sp^{\sm}$ of very small spectra.
We have a suspension spectrum functor $\Sigma_{+}^{\infty} \colon \Top\to \Sp^{\sm}$, and we denote by $\cF_{\cC,\cU/2,\cB} \colon \UBC\to \Top$ the canonical forgetful functor.

\begin{lem}\label{qriofrefqewewfqf}
We have an equivalence of functors
\begin{equation}\label{fjo3ifjoi34fj34f34f4f}
F^{\%}(-)\simeq (  F(*) \wedge \Sigma^{\infty}_{+}(- ))\circ \cF_{\cC,\cU/2,\cB}
\end{equation}
from $\UBC$ to $\bC$.
In particular, $F^{\%}$ is homotopy invariant and satisfies open excision.
\end{lem}

\begin{proof}
 Since $F$ is homotopy invariant,  the projection $\Delta^{n}\to *$ induces an equivalence
 $$F\circ q \to \const(F(*))\simeq  \const(F(*) \wedge \Sigma^{\infty}_{+}(*) ) \, .$$
 Using the  equivalence $ \colim_{(\Delta^n \to X)} \Sigma^{\infty}_{+}(*)\simeq \Sigma^{\infty}_{+}(X)$ (which is natural in $X$) and the fact that $\wedge$ commutes with colimits, we get  the equivalence \eqref{fjo3ifjoi34fj34f34f4f}.
  
The functor $\Sigma^{\infty}_{+}(-):\Top\to \Sp^{sm}$ is homotopy invariant and satisfies open excision.
This implies that the functor $F^{\%}$ is homotopy invariant and satisfies open excision.
\end{proof}

\begin{kor}\label{ergoijqrofirfweqwddq}
The functor  $(F^{\%})^{\lf} \colon \UBC \to \bC$ is an open local homology theory, it is additive, {and it is invariant under coarsening.}
\end{kor}

\begin{proof}
We let $ (F(*)\wedge \Sigma_{+}^{\infty})^{\lf} \colon \TopBorn\to \bC$ be the open local homology theory associated to the object $F(*)$ of $\bC$ \cite[Ex.~7.40]{buen}. 
By Lemma~\ref{qriofrefqewewfqf} we have 
\begin{equation}
\label{eq_formula_lf_point}
(F^{\%})^{\lf}\simeq (F(*)\wedge \Sigma_{+}^{\infty})^{\lf}\circ \cF_{\cC,\cU/2}\,.
\end{equation}
  The right-hand side is an open local homology theory by Lemma~\ref{fiweofwefewfewfewf}. 
Additivity of $(F^{\%})^{\lf}$ follows {from Lemma~\ref{lem_locally_finite_excisive} since $F^{\%}$ satisfies excision by Lemma~\ref{qriofrefqewewfqf}.}
Finally, by~\eqref{eq_formula_lf_point} the functor $(F^{\%})^{\lf}$ is completely independent of the coarse structure and hence in particular invariant under coarsening.
\end{proof}

%


\begin{lem}\label{egiojerogrgreg}
If $F$ satisfies the assumptions stated in Lemma \ref{feuhweifwfewwf},  and $X$ is a countable, locally finite, finite-dimensional  simplicial complex equipped with the metric structures, then the natural morphism
$$(F^{\%})^{\lf}(X)\to F^{\lf}(X)$$
is an equivalence.
\end{lem}

\begin{proof}
The argument is the same as for Lemma~\ref{feuhweifwfewwf}.
{It is an induction over the dimension of $X$. In the zero-dimensional case we just need additivity of $F^{\lf}$ and $(F^{\%})^{\lf}$. Additivity of $F^{\lf}$ follows from Lemma~\ref{lem_locally_finite_excisive} and additivity of $(F^{\%})^{\lf}$ was shown in Corollary~\ref{ergoijqrofirfweqwddq}.
}

{For the induction step we need that the functors $(F^{\%})^{\lf}$ and $F^{\lf}$ are homotopy invariant, invariant under coarsenings and excisive for the decompositions of simplicial complexes as depicted on Page~\pageref{fig543er}. That $F^{\lf}$ has these properties was already explained in the proof of Lemma~\ref{feuhweifwfewwf}.}

Homotopy invariance and invariance under coarsening of $(F^{\%})^{\lf}$ is shown in Corollary~\ref{ergoijqrofirfweqwddq}. Only the required excisiveness of $(F^{\%})^{\lf}$ is not completely obvious: here we have to use homotopy invariance of $(F^{\%})^{\lf}$ to transform the closed decompositions into open ones.
\end{proof}

\begin{rem}
The conclusion of Lemma~\ref{egiojerogrgreg} is also true if we assume that $F$ satisfies open excision instead of closed excision. In this case we have to slightly modify the proof {similarly as in} Remark~\ref{jknerinwefsd}.
\hB
\end{rem}

Let $E\colon \BC \to \bC$ be a coarse homology theory, and let
$X$ be a uniform bornological coarse space.
In the following we omit to write the forgetful functor $\cF_{\cC,\cU/2,\cB}$ in front of $\Sigma^{\infty}_{+}$ in order to simplify the notation.
\begin{prop}\label{fifowefweewfwef}
Assume:
\begin{enumerate}
\item $\bC$ is complete.
\item $E$ is additive.
\item $X$ is homotopy equivalent in $\UBC$ to a   locally finite, finite-dimensional simplicial complex equipped with the metric structures.
\end{enumerate}
Then we have an equivalence
$$(\Sigma E(*)\wedge \Sigma^{\infty}_{+})^{\lf}( X)\simeq
E\cO^{\infty}(X)\, .$$
\end{prop}

\begin{proof}
We first observe that $\Sigma E(*)\simeq E\cO^{\infty}(*)$.
Then we just combine Proposition \ref{fiuwofewfewfewf}, Lemma  \ref{egiojerogrgreg} with $F \coloneqq E\cO^{\infty}$, and \eqref{fjo3ifjoi34fj34f34f4f}.
\end{proof}

Note that the functor
$X\mapsto (\Sigma E(*)\wedge \Sigma^{\infty}_{+})^{\lf}(X)$ is naturally defined on $\Top\Born$ and is locally finite, homotopy invariant (in the sense of $\TopBorn$, i.e., for proper homotopies {which are not necessarily uniform}), and satisfies open excision. Therefore the functor
$X\mapsto E\cO^{\infty}(X)$ for spaces $X$ in $\UBC$ (which are   homotopy equivalent in the sense of $\UBC$   to a  locally finite, finite-dimensional simplicial complexes equipped with the metric structures), also has these stronger homological properties. In particular:

Let $E \colon \BC \to \bC$ be a coarse homology theory. Furthermore,
let $X,X^{\prime}$ be in  $\UBC$ and $f,g \colon X\to X^{\prime}$ be morphisms in 
 $\UBC$.
\begin{kor}\label{rgigoeri34345345345}
Assume:
\begin{enumerate}
\item $\bC$ is complete.
\item $E$ is additive.
\item $X$ and $X^{\prime}$ are homotopy   equivalent in $\UBC$ to locally finite and finite-dimensional simplicial complexes equipped with the metric structures.
\item $F_{\cC,\cU/2}(f)$ and $F_{\cC,\cU/2}(g)$  are properly homotopic (there is a homotopy $[0,1]\times X\to X^{\prime}$ which is continuous and proper after forgetting the coarse and uniform structures).
\end{enumerate}
Then $E\cO^{\infty}(f)$ is equivalent to $E\cO^{\infty}(g)$.
\end{kor}

\section{Comparison of coarse homology theories}

In ordinary homotopy theory a transformation between spectrum-valued homology theories which induces an equivalence on a point is an equivalence at least on all $CW$-complexes. In the present section we consider an analogous statement for coarse homology theories.

 Assume that we have a transformation $E \to E^\prime$ of $\bC$-valued coarse homology theories  which induces an equivalence $E(\ast) \to E^\prime(\ast)$. In this section we provide sufficient conditions on a  bornological coarse space $X$ and on the theories $E$ and $E^\prime$ which imply that $E(X) \to E^\prime(X)$ is   an equivalence. The main result is formulated in Theorem~\ref{fewoihjowefwfewfwef}.

Let $X$ be a bornological coarse space.
The following definitions are from \cite[Def.~7.75]{buen}, \cite[Def.~7.77]{buen}.

\begin{ddd}\mbox{}
\begin{enumerate}
\item $X$ has strongly   bounded geometry if it has the minimal bornology compatible with the coarse structure and for every coarse entourage $U$ of $X$  the number of points in $U$-bounded subsets of $X$ is uniformly bounded.
\item $X$ has bounded geometry if it is   equivalent to a bornological coarse space with strongly  bounded geometry.
\end{enumerate}
\end{ddd}

Let $X$ be a bornological coarse space, and let $E \colon \BC\to \bC$ be a strong 
coarse homology theory.
\begin{prop}\label{rgriugeroigregregerg}
Assume:
\begin{enumerate}
\item $\bC$ is complete.
\item\label{ergioergwerfrefwefwerf} $E$ is additive.
\item $X$ has 
bounded geometry.
\end{enumerate}

Then we have an equivalence
$$((\Sigma E(*)\wedge \Sigma^{\infty}_{+})^{\lf}\circ F_{\cC,\cU/2})\bP^{o}(X)\simeq E\cO^{\infty}\bP(X)\, .$$
\end{prop}

\begin{proof}
Since $(\Sigma E(*)\wedge \Sigma^{\infty}_{+})^{\lf}\circ F_{\cC,\cU/2}$ is  by  Lemma \ref{fiweofwefewfewfewf}    an open local homology theory, it can be composed with the open version $\bP^{o}$.
Since both sides of the equivalence are coarsely invariant we can assume that $X$ is a    bornological coarse space of   strongly   
 bounded geometry.  Then for every entourage $U$ of $X$ the complex
$P_{U}(X)$ is a   locally finite, finite-dimensional simplicial complex. 
Hence by Proposition \ref{fifowefweewfwef} we get an equivalence 
$$(\Sigma E(*)\wedge \Sigma^{\infty}_{+})^{\lf}(
P_{U}(X)
)\simeq E\cO^{\infty}(P_{U}(X))\, .$$
Forming the colimit over the entourages $U$ of $X$  and using \eqref{oirjeorgergregr}  and its open version we get the claimed equivalence
\[
((\Sigma E(*)\wedge \Sigma^{\infty}_{+})^{\lf}\circ F_{\cC,\cU/2})\bP^{o}(X)\simeq E\cO^{\infty}\bP^{o}(X)\stackrel{\text{Lem.}~\ref{vweoijiorjovbwvdfsv}}{\simeq} E\cO^{\infty}\bP(X)\,.
\]
In the second equivalence   Lemma \ref{vweoijiorjovbwvdfsv} can be applied (with $E\cO^{\infty}$ in place of $E$) since $E\cO^{\infty}$ is a closed and open local homology theory
by Lemma \ref{fewoiiofwefewf} and since $E$ is strong.
\end{proof}


Let $E\to E^{\prime}$ be a transformation between strong $\bC$-valued coarse homology theories, and let $X$ be a bornological coarse space.

\begin{theorem}\label{fewoihjowefwfewfwef}
Assume:
\begin{enumerate}
\item $\bC$ is complete.
\item $E$ and $E^{\prime}$ are additive.
\item \label{iuheiwofwefewfew}$E(*)\to E^{\prime}(*)$ is an equivalence.
\item $X$ is of bounded geometry.
\item\label{fwu28r224f} The coarse assembly maps $\mu_{E,X}$ and $\mu_{E^{\prime},X}$ are equivalences {(Definition~\ref{fiwjfofewfewfwefefweffw}).}
\end{enumerate}
Then $E(X)\to E^{\prime}(X)$ is an equivalence.
\end{theorem}

\begin{proof}
By an inspection of the arguments going into the proof of Proposition \ref{rgriugeroigregregerg}
one checks that the asserted equivalence is natural in $E$.  This gives the left commuting square in the following diagram. Similarly for the right square we use that  the coarse assembly map is natural in $E$.
 $$\xymatrix{
((\Sigma E(*)\wedge \Sigma^{\infty}_{+})^{\lf}\circ F_{\cC,\cU/2})\bP^{o}(X)\ar[rr]^-{\text{Prop.}~\ref{rgriugeroigregregerg}}_-{\simeq}\ar[d]^-{\simeq} & & E\cO^{\infty}\bP (X) \ar[r]_-{\simeq}^-{\mu_{E,X}}\ar[d] & \Sigma E(X)\ar[d]\\
((\Sigma E^{\prime}(*)\wedge \Sigma^{\infty}_{+})^{\lf}\circ F_{\cC,\cU/2})\bP^{o}(X)\ar[rr]^-{\text{Prop.}~\ref{rgriugeroigregregerg}}_-{\simeq} & & E^{\prime}\cO^{\infty}\bP (X)\ar[r]_-{\simeq}^-{\mu_{E^{\prime},X}} & \Sigma E^{\prime}(X)
}$$
The left vertical morphism is an equivalence by Condition \ref{iuheiwofwefewfew}. We conclude that the right vertical morphism is an equivalence, too.
\end{proof}

We can use Theorems \ref{woifowfwewfewf},  \ref{feoijofwefewfewf} and \ref{wrefiowefwfewfewfewfw} in order to check Condition \ref{fwu28r224f} in the statement of Theorem \ref{fewoihjowefwfewfwef}.

\bibliographystyle{alpha}
\bibliography{ass3_bibliography}

\end{document}